\documentclass[11pt]{article}

\usepackage[nocompress]{cite}
\usepackage[dvips]{graphics,color}
\usepackage{amssymb}
\usepackage{amsfonts}
\usepackage{graphicx}
\usepackage{amsthm}
\usepackage{subfigure}
\usepackage{algorithm}
\usepackage{algorithmic}
\usepackage{geometry}
\usepackage{epsfig}

\begin{document}

\title{On Inner Iterations in the Shift-Invert Residual Arnoldi Method and the
Jacobi--Davidson Method\footnote{Supported
by National Basic Research Program
of China 2011CB302400 and the
National Science Foundation of China (No. 11071140).}}

\author{Zhongxiao Jia\thanks{Department of Mathematical Sciences,
Tsinghua University, Beijing 100084, People's Republic of China,
{\sf jiazx@tsinghua.edu.cn}.} \and Cen Li\thanks{Department of Mathematical
Sciences, Tsinghua University, Beijing 100084, People's Republic of China,
{\sf licen07@mails.tsinghua.edu.cn}.}}
\date{}
\maketitle

\newtheorem{Def}{\bf Definition}
\newtheorem{Exm}{\bf Example}
\newtheorem{Thm}{\bf Theorem}
\newtheorem{Lem}{\bf Lemma}
\newtheorem{Cor}{\bf Corollary}
\newtheorem{Alg}{\bf Algorithm}
\newtheorem{Exp}{\bf Experiment}
\newtheorem{Ass}{\bf Assumption}
\newtheorem{Rem}{\bf Remark}
\newtheorem{Pro}{\bf Proposition}

\newcommand{\Span}{\mathrm{span}}
\newcommand{\xp}{\mathbf{x}_\perp}
\newcommand{\Xp}{X_\perp}
\newcommand{\sep}{\mathrm{sep}}
\newcommand{\imag}{\mathrm{i}}
\newcommand{\bfx}{\mathbf{x}}
\newcommand{\bfy}{\mathbf{y}}
\newcommand{\bft}{\mathbf{t}}
\newcommand{\bfu}{\mathbf{u}}
\newcommand{\bfr}{\mathbf{r}}
\newcommand{\bff}{\mathbf{f}}
\newcommand{\bfv}{\mathbf{v}}
\newcommand{\bfw}{\mathbf{w}}
\newcommand{\bfz}{\mathbf{z}}
\newcommand{\bfg}{\mathbf{g}}
\newcommand{\bfa}{\mathbf{a}}
\newcommand{\bfb}{\mathbf{b}}
\newcommand{\bfc}{\mathbf{c}}
\newcommand{\bfd}{\mathbf{d}}
\newcommand{\bfh}{\mathbf{h}}
\newcommand{\bfB}{\mathbf{B}}
\newcommand{\bfI}{\mathbf{I}}
\newcommand{\bfA}{\mathbf{A}}
\newcommand{\bfV}{\mathbf{V}}
\newcommand{\bfH}{\mathbf{H}}
\newcommand{\bfP}{\mathbf{P}}
\newcommand{\bfU}{\mathbf{U}}
\newcommand{\bfX}{\mathbf{X}}
\newcommand{\bfL}{\mathbf{L}}
\newcommand{\bfQ}{\mathbf{Q}}
\newcommand{\bfR}{\mathbf{R}}
\newcommand{\bfE}{\mathbf{E}}
\newcommand{\bfM}{\mathbf{M}}
\newcommand{\bfSig}{\mathbf{\Sigma}}

\begin{abstract}
Using a new analysis approach, we establish a general
convergence theory of the Shift-Invert Residual Arnoldi (SIRA)
method for computing a simple eigenvalue nearest to a given target $\sigma$ and
the associated eigenvector. In SIRA, a subspace expansion vector at each step
is obtained by solving a certain inner linear system. We prove that the
inexact SIRA method mimics the exact SIRA well, that is,
the former uses almost the same outer iterations to achieve
the convergence as the latter does if all the inner linear systems
are iteratively solved with {\em low} or {\em modest} accuracy during
outer iterations. Based on the theory, we design practical stopping criteria
for inner solves. Our analysis is on one step expansion of subspace and the approach
applies to the Jacobi--Davidson (JD) method with the fixed target $\sigma$ as well,
and a similar general convergence theory is obtained for it.
Numerical experiments confirm our theory and demonstrate that the inexact SIRA and JD
are similarly effective and are considerably superior to the inexact SIA.

\smallskip

{\bf Keywords.} Subspace expansion, expansion vector, inexact,
low or modest accuracy, the SIRA method, the JD method,
inner iteration, outer iteration.

\smallskip

{\bf AMS subject classifications.} 65F15, 15A18, 65F10.
\end{abstract}

\section{Introduction}
Consider the large and possibly sparse matrix eigenproblem
\begin{equation}
\bfA\bfx=\lambda\bfx, \label{problem}
\end{equation}
with $\bfA\in\mathcal{C}^{n\times n}$, the 2-norm $\|\bfx\|=1$
and the eigenvalues labeled as
\begin{eqnarray*}
0<|\lambda_1-\sigma|<|\lambda_2-\sigma|\leq\cdots\leq|\lambda_n-\sigma|
\end{eqnarray*}
for a given target $\sigma\in\mathcal{C}$. We are interested in the eigenvalue
$\lambda_1$ closest to the target $\sigma$ and/or the associated eigenvector
$\bfx_1$. We denote $(\lambda_1,\bfx_1)$ by $(\lambda,\bfx)$ for simplicity.
A number of numerical methods \cite{bai2000templates,parlett1998symmetric,
saad1992eigenvalue,vandervorst2002eigenvalue,stewart2001eigensystems} are
available for solving this kind of problems. The Residual Arnoldi (RA)
method and Shift-Invert Residual Arnoldi (SIRA) method are new ones that have
their origins in the Jacobi--Davidson (JD) method \cite{sleijpen2000jacobi}.
RA was initially proposed by van der Vorst and Stewart in 2001;
see \cite{leestewart07}. The methods
were then studied and developed by Lee \cite{lee2007residual} and Lee and Stewart
\cite{leestewart07}. We briefly describe RA now.

Given a starting vector $\bfv_1$ with $\|\bfv_1\|=1$, suppose an orthonormal
$\bfV_m=(\bfv_1,\ldots,\bfv_m)$ has been constructed by
the Arnoldi process. Then the columns of $\bfV_m$ form a
basis of the $m$-dimensional Krylov subspace $\mathcal{K}_m(\bfA,\bfv_1)
=span\{\bfv_1,\bfA\bfv_1,\ldots,\bfA^{m-1}\bfv_1\}$,
and the next basis vector $\bfv_{m+1}$ is obtained by orthogonalizing
$\bfA\bfv_m$ against $\bfV_m$. Let $(\tilde\lambda,\bfy)$ be the
candidate Ritz pair of $\bfA$ for a desired eigenpair of $\bfA$ with
respect to $\mathcal{K}_m(\bfA,\bfv_1)$, and define
the residual $\bfr=\bfA\bfy-\tilde\lambda\bfy$. Then
the RA method orthogonalizes $\bfr$ against $\bfV_m$ to get the
next basis vector, which, in exact arithmetic, is just $\bfv_{m+1}$
obtained by the Arnoldi process \cite{lee2007residual,leestewart07}. So
the Arnoldi method is mathematically equivalent to the RA method. However,
van der Vorst and Stewart discovered a striking phenomenon
that $\bfr$ in the RA method may allow much larger errors or perturbations
than $\bfA\bfv_m$ in the Arnoldi method.

The Shift-Invert Arnoldi (SIA) method is the Arnoldi method applied to
the shift-invert matrix $\bfB=(\bfA-\sigma\bfI)^{-1}$ and
finds a few eigenvalues nearest to $\sigma$ and the associated eigenvectors.
It computes $\bfv_{m+1}$ by orthogonalizing $\bfu=\bfB\bfv_m$ against
$\bfV_m$, whose columns now form a basis of $\mathcal{K}_m(\bfB,\bfv_1)$.
So at step $m$ one has to solve the linear system
\begin{equation}\label{siainner}
(\bfA-\sigma\bfI)\bfu=\bfv_m
\end{equation}
for $\bfu$. The SIRA method \cite{lee2007residual,leestewart07}
is an alternative of the RA method applied to $\bfB$.
At each step one has to solve the linear system
\begin{equation}
\label{inneriteration}(\bfA-\sigma \bfI)\bfu=\bfr
\end{equation}
for $\bfu$, where $\bfr=\bfA\bfy-\nu\bfy$ is the residual of the current
approximate eigenpair $(\nu,\bfy)$ obtained by SIRA. Then the SIRA method computes
the next basis vector $\bfv_{m+1}$ by orthogonalizing $\bfu$ against $\bfV_m$.
A mathematical difference between SIA and SIRA is that the SIA method computes
Ritz pairs of the shift-invert $\bfB$ with respect to $\mathcal{K}_m(\bfB,\bfv_1)$
and recovers an approximation to $(\lambda,\bfx)$, while the SIRA
method computes the Ritz pairs of the original $\bfA$ with respect to
the same $\mathcal{K}_m(\bfB,\bfv_1)$ and gets an approximation to $(\lambda,\bfx)$.
So SIA and SIRA generally obtain different approximations to $(\lambda,\bfx)$
with respect to the same subspace $\mathcal{K}_m(\bfB,\bfv_1)$.

However, for large (\ref{inneriteration}), only iterative solvers are generally
viable. This leads to the inexact SIRA, an inner-outer iterative method,
built-up by outer iteration as the eigensolver and inner iteration as the
solver of (\ref{inneriteration}). Inexact eigensolvers have attracted much
attention over the last two decades, and among them
inexact SIA type methods \cite{simoncini2003ia, simoncini2005ia,spence2009ia,xueelman10}
are closely related to the work in the current paper.
Central concerns on all inexact eigensolvers are how the accuracy of
inner iterations ensures and affects the convergence of outer iterations and how
to choose the accuracy requirements of inner iterations so that each inexact eigensolver
mimics its corresponding exact counterpart very well in the sense
that the two eigensolvers use almost the same or very comparable outer iterations
to achieve the convergence.

The JD method with fixed or variable targets \cite{sleijpen2000jacobi} is a
very popular inexact eigensolver, in which a correction equation (inner linear system)
is solved iteratively at each outer iteration; see, e.g.,
\cite{bai2000templates,vandervorst2002eigenvalue,stewart2001eigensystems}
and more recent \cite{hochstenbachnotay09,notay02,stathopoulos,voss07}.
Hitherto, however, there has been no result on the accuracy requirement of inner
iterations involved in the standard JD method. Existing work only
focuses on the simplified (or single-vector) JD method without subspace acceleration.
One hopes that the results on the accuracy requirement of inner
iterations developed for the simplified JD may help understand the standard JD.
Nevertheless, such treatment may be too inaccurate and far from the essence
of the standard JD. As is well known,
the standard JD is much more complicated than the simplified JD,
and the convergence of its outer iterations is much more involved; see
\cite{jia2001analysis} and also
\cite{bai2000templates,vandervorst2002eigenvalue,stewart2001eigensystems}
for details. Therefore, the standard JD method lacks a general theory on inner
iterations, and a rigorous and insightful analysis is necessary and very
appealing.

For the inexact SIA method, Simoncini \cite{simoncini2003ia} has established a
relaxation theory on the accuracy requirements of inner iterations of (\ref{siainner})
as $m$ increases. She proved that the accuracy of approximate solution of
(\ref{siainner}) should be very high initially
and is relaxed as the approximate eigenpairs start converging.
Freitag and Spence \cite{spence2009ia} have extended Simoncini's
relaxation theory to the inexact implicitly restarted Arnoldi method.
Xue and Elman~\cite{xueelman10} have made a refined analysis on the relaxation
strategy. 
So it may be very costly to implement the inexact SIA type methods.

For the SIRA method, it has been reported by Lee \cite{lee2007residual} and
Lee and Stewart \cite{leestewart07} that when the accuracy of approximate
solutions of (\ref{inneriteration}) is low or modest at each
step, the method may still work well.
Lee and Stewart \cite{leestewart07} have made some analysis on the RA and SIRA methods
but they did not derive any quantitative
and explicit bounds for the accuracy requirements of inner iterations.

In this paper, we take a different approach from that in
\cite{lee2007residual,leestewart07} to giving
a rigorous one-step analysis of the inexact SIRA method
and establish a general and
quantitative theory of the accuracy requirements of inner iterations.
Our analysis approach applies to the JD method with the fixed target $\sigma$
as well. We first show that the exact SIRA and JD methods are mathematically
equivalent. We then focus on a detailed
quantitative analysis of the inexact SIRA and JD methods.
Let $\varepsilon$ be the relative error of
the approximate solution of the inner linear system.
We prove that a modestly small $\varepsilon$, e.g.,
$\varepsilon \in [10^{-4},10^{-3}]$, is generally enough to make
the inexact SIRA and JD use almost the same outer iterations
as the exact ones to achieve the convergence.
As a result, one only needs to solve all inner linear systems
with low or modest accuracy in the inexact SIRA and the JD methods,
and both methods are expected to
be considerably more effective than the inexact SIA method.
We should point out that our work is locally an one step analysis.
A global analysis involving subspaces accumulating
all previous perturbations is much harder and seems impossible.
Actually, an one step local analysis is typical in the field of inexact eigensolvers,
and it indeed sheds lights on the behavior of the inexact solvers.

The paper is organized as follows. In Section~\ref{sec:sira_vs_jd},
we review the SIRA and JD methods and show the equivalence of two exact
versions. In Section~\ref{sec:eps}, we derive some
relationships between $\varepsilon$ and subspace expansions
and show that the inexact SIRA and methods
behave very similar when their respective inner linear systems are solved
with the same accuracy. In Section~\ref{sec:teps},
we consider subspace improvement and the selection of $\varepsilon$
and prove that the inexact SIRA mimics
the exact SIRA very well when $\varepsilon$ is modestly small at
all steps. In Section~\ref{issue},
we consider some practical issues and design practical stopping criteria for
inner solves in the inexact SIRA and JD.
In Section~\ref{numer}, we report numerical experiments to confirm
our theory and the considerable superiority of the
inexact SIRA and JD algorithms to the inexact SIA algorithm. Meanwhile, we
show that the inexact SIRA and JD are similarly effective.
Finally, we conclude the paper and point out future work
in Section~\ref{concl}.

Throughout the paper, denote by $\|\cdot\|$ the 2-norm of a vector or matrix,
by $\bfI$ the identity matrix with the order clear from the context,
by the superscript $H$ the conjugate
transpose of a vector or matrix, and by
$\kappa(\bfQ)=\|\bfQ\|\|\bfQ^{-1}\|$ the condition number of a nonsingular
matrix $\bfQ$. We measure the
distance between a nonzero vector $\bfy$ and a subspace $\mathcal{V}$
by
\begin{equation}
\sin\angle(\mathcal{V},\bfy)=\frac{\|(\bfI-\bfP_\bfV)\bfy\|}{\|\bfy\|}
=\frac{\|\bfV_\perp^H\bfy\|}{\|\bfy\|}, \label{sinedef}
\end{equation}
where $\bfP_\bfV$ is the orthogonal projector onto $\mathcal{V}$
and the columns of $\bfV_\perp$ form
an orthonormal basis of the orthogonal complement of $\mathcal{V}$.

\section{Equivalence of the exact SIRA and JD methods}
\label{sec:sira_vs_jd}

Algorithms~\ref{alg:sira}--\ref{alg:jd} describe
the SIRA algorithm and the JD algorithm with the fixed target $\sigma$,
respectively (for brevity we drop iteration subscript).
Comparing them, we observe that the only
seemingly differences between them are the linear systems to be solved (step 4)
and the expansion vectors to be orthogonalized against the
initial subspace $\mathcal{V}$. In fact, they are equivalent, as the following
theorem shows.

\begin{algorithm}\caption{SIRA method with the target $\sigma$}\label{alg:sira}
\begin{algorithmic}
\STATE{Given the target $\sigma$ and a user-prescribed convergence tolerance
$tol$, suppose the columns of $\bfV$ form an orthonormal basis of an initial subspace
$\mathcal{V}$.}
\REPEAT
\STATE{1. Compute the Rayleigh quotient $\bfH=\bfV^H\bfA\bfV$.}
\STATE{2. Let $(\nu,\bfz)$ be an eigenpair of $\bfH$, where $\nu\cong\lambda$.}
\STATE{3. Compute the residual $\bfr_S=\bfA\bfy-\nu\bfy$,
where $(\nu,\bfy)=(\nu,\bfV\bfz)$.}
\STATE{4. Solve the linear system
\begin{equation}\label{lssira}
(\bfA-\sigma \bfI)\bfu=\bfr_S.
\end{equation}}
\STATE{5. Orthonormalize $\bfu$ against $\bfV$ to get $\bfv$.}
\STATE{6. Expand the subspace as $\bfV=\left[\begin{array}{cc}
\bfV & \bfv\end{array}\right]$ and update $\bfH$.}
\UNTIL{$\|\bfr_S\|<tol$.}
\end{algorithmic}
\end{algorithm}
\begin{algorithm}\caption{Jacobi--Davidson method with the
fixed target $\sigma$}\label{alg:jd}
\begin{algorithmic}
\STATE{Given the target $\sigma$ and a user-prescribed convergence tolerance $tol$,
suppose the columns of $\bfV$ form an orthonormal basis of an initial subspace
$\mathcal{V}$.}
\REPEAT
\STATE{1. Compute the Rayleigh quotient $\bfH=\bfV^H\bfA\bfV$.}
\STATE{2. Let $(\nu,z)$ be an eigenpair of $\bfH$, where $\nu\cong\lambda$.}
\STATE{3. Compute the residual $\bfr_J=\bfA\bfy-\nu \bfy$, where
$(\nu,\bfy)=(\nu,\bfV\bfz)$.}
\STATE{4. Solve the correction equation for $\bfu\perp \bfy$,
\begin{equation}
\label{ls_jd}(\bfI-\bfy\bfy^H)(\bfA-\sigma \bfI)(\bfI-\bfy\bfy^H)\bfu=-\bfr_J.
\end{equation}}
\STATE{5. Orthonormalize $\bfu$ against $\bfV$ to get $\bfv$.}
\STATE{6.  Expand the subspace as $\bfV=\left[\begin{array}{cc}
\bfV & \bfv\end{array}\right]$ and update $\bfH$.}
\UNTIL{$\|\bfr_S\|<tol$.}
\end{algorithmic}
\end{algorithm}
\begin{Thm}\label{equiva}
For the same initial $\mathcal{V}$, if $\sigma\not=\nu$, then
the SIRA method and the JD method are mathematically
equivalent when inner linear systems
{\rm (\ref{lssira})} and {\rm (\ref{ls_jd})} are solved exactly.
\end{Thm}

\begin{proof}
For the same initial ${\cal V}$, the two
methods share the same $\bfH$, $\nu$ and $\bfy$, leading to the same
$\bfr_S$ and $\bfr_J$. Let $\bfu_S$ and $\bfu_J$ be the exact solutions of
(\ref{lssira}) and (\ref{ls_jd}), respectively. Since $\bfB=(\bfA-\sigma\bfI)^{-1}$,
we get
\begin{eqnarray}
\label{u_sira}\bfu_S=\bfB\bfr_S=(\sigma-\nu)\bfB\bfy+\bfy.
\end{eqnarray}
From (\ref{ls_jd}), we have
\begin{equation}
\label{ls_jd_0}(\bfA-\sigma\bfI)\bfu_J=\left(\bfy^H(\bfA-\sigma\bfI)
\bfu_J\right)\bfy-\bfr_J=\gamma\bfy-(\bfA-\sigma\bfI)\bfy,
\end{equation}
where $\gamma=\bfy^H(\bfA-\sigma \bfI)\bfu_J-\sigma+\nu$.
Premultiplying two sides of (\ref{ls_jd_0}) by $\bfB $, we obtain
\begin{equation}
\label{u_jd}\bfu_J=\gamma\bfB \bfy-\bfy.
\end{equation}
Since $\bfu_J\perp \bfy$, we get $\gamma=\frac{1}{\bfy^H\bfB \bfy}$.
Since $\bfy\in\mathcal{V}$, we have $(\bfI-\bfP_\bfV)\bfy=\bf0$.
So from (\ref{u_sira}) and (\ref{u_jd}), we get
\begin{equation}
(\bfI-\bfP_\bfV)\bfB \bfy=\frac{1}{\sigma-\nu}
(\bfI-\bfP_\bfV)\bfu_S=\frac{1}{\gamma}(\bfI-\bfP_\bfV)\bfu_J.
\label{expand}
\end{equation}
Note that $(\bfI-\bfP_\bfV)\bfu_S$ and $(\bfI-\bfP_\bfV)\bfu_J$
(after normalization) are the subspace expansion vectors in SIRA
and JD, respectively. The two methods generate the same subspace
in the next iteration and $(\nu,\bfy)$ obtained by them are thus identical.
\end{proof}

From (\ref{ls_jd_0}), define
$$
\label{rj'}\bfr_J'=\bfA\bfy-(\sigma+\gamma)\bfy,
$$
where
$$
\gamma=\bfy^H(\bfA-\sigma \bfI)\bfu_J-\sigma+\nu=\frac{1}{\bfy^H\bfB \bfy}.
\label{gamma}
$$
Then (\ref{ls_jd_0}) and thus (\ref{ls_jd}) become
\begin{equation}
\label{ls_jd_1}(\bfA-\sigma \bfI)\bfu=\bfr_J', \label{mjd}
\end{equation}
whose solution is $-\bfu_J$ and is the same as $\bfu_J$ up to the sign $-1$.
So mathematically, hereafter we use (\ref{mjd}) as the inner linear system in
the JD method. Since $\bfy^H\bfB \bfy$ approximates the eigenvalue
$\frac{1}{\lambda-\sigma}$ of $\bfB $, $\gamma+\sigma=\frac{1}
{\bfy^H\bfB \bfy}+\sigma$ approximates $\lambda$. So $\bfr_J'$
is a residual associated with the desired eigenpair $(\lambda,\bfx)$, just
like $\bfr_S$ in (\ref{lssira}).

\section{Relationships between the accuracy of inner iterations and subspace
expansions}\label{sec:eps}

We observe that (\ref{lssira})
and (\ref{ls_jd_1}) fall into the category of
\begin{equation}
\label{ls_unified}(\bfA-\sigma \bfI)\bfu=\alpha_1\bfy+\alpha_2
(\bfA-\sigma \bfI)\bfy,\label{alpha12}
\end{equation}
where specifically $\alpha_1=\sigma-\nu$
and $\alpha_2=1$ in SIRA and $\alpha_1=-\frac{1}{\bfy^H\bfB \bfy}$
and $\alpha_2=1$ in JD.
The exact solution $\bfu$ of (\ref{alpha12}) is
\begin{equation}
\label{u_unified}\bfu=\alpha_1\bfB \bfy+\alpha_2 \bfy.
\end{equation}
Since $(\bfI-\bfP_{\bfV})\bfy=\bf0$, the (unnormalized) subspace expansion vector is
$(\bfI-\bfP_\bfV)\bfu=\alpha_1(\bfI-\bfP_\bfV)\bfB \bfy$.
Let $\tilde{\bfu}$ be an approximate solution of
(\ref{ls_unified}), whose relative error is defined by
\begin{equation}
\varepsilon=\frac{\|\tilde{\bfu}-\bfu\|}{\|\bfu\|}.
\label{errorsol}
\end{equation}
Then we can write
$$
\tilde{\bfu}=\bfu+\varepsilon\|\bfu\|\bff
\label{errorf}
$$
with $\bff$ the normalized error direction vector.
So we get
\begin{equation}
\label{IPtu}(\bfI-\bfP_\bfV)\tilde{\bfu}=(\bfI-\bfP_\bfV)
\bfu+\varepsilon\|\bfu\|\bff_\perp.
\end{equation}
where
\begin{equation}
\label{fperp} \bff_\perp=(\bfI-\bfP_\bfV)\bff.
\end{equation}
Define
\begin{eqnarray}
\label{tv_and_v}\tilde{\bfv}=\frac{(\bfI-\bfP_\bfV)\tilde{\bfu}}
{\|(\bfI-\bfP_\bfV)\tilde{\bfu}\|},\ \ \ \bfv=\frac{(\bfI-\bfP_\bfV)\bfu}
{\|(\bfI-\bfP_\bfV)\bfu\|},
\end{eqnarray}
which are the normalized subspace expansion vectors in the inexact and exact methods,
respectively. We measure the difference between $(\bfI-\bfP_\bfV)\tilde{\bfu}$
and
$(\bfI-\bfP_\bfV)\bfu$ by the relative error
\begin{equation}
\label{varepsilon}\tilde{\varepsilon}=\frac{\|(\bfI-\bfP_\bfV)
\tilde{\bfu}-(\bfI-\bfP_\bfV)\bfu\|}{\|(\bfI-\bfP_\bfV)\bfu\|}
\end{equation}
or by $\sin\angle(\tilde{\bfv},\bfv)$. Two quantities
$\tilde{\varepsilon}$ and $\sin\angle(\tilde{\bfv},\bfv)$
are two valid measures for the difference.
Next we establish a relationship between $\tilde\varepsilon$ and
$\sin\angle(\tilde\bfv,\bfv)$, which will be used
in proving our final result in this paper.

\begin{Lem}\label{lemma1}
With the notations defined above, it holds that
\begin{equation}
\label{sin_tv_v_final}\sin\angle(\tilde{\bfv},\bfv)
=\tilde{\varepsilon}\sin\angle(\tilde{\bfv},\bff_\perp).
\end{equation}
\end{Lem}

\begin{proof}
Let $\bfU_\perp$ be an orthonormal basis of the orthogonal complement of
$span\left\{(\bfI-\bfP_\bfV)\tilde{\bfu}\right\}$ with respect to ${\cal C}^n$.
Since $\bfU_\perp^H(\bfI-\bfP_\bfV)\tilde{\bfu}=\mathbf{0}$, by definition (\ref{sinedef})
we get
\begin{eqnarray}
\sin\angle(\tilde{\bfv},\bfv)
\nonumber&=&\sin\angle\left((\bfI-\bfP_\bfV)\tilde{\bfu},(\bfI-\bfP_\bfV)
\bfu\right)\\
\nonumber&=&\frac{\left\|\bfU_\perp^H(\bfI-\bfP_\bfV)\bfu\right\|}
{\|(\bfI-\bfP_\bfV)\bfu\|}\\
\nonumber&=&\frac{\left\|\bfU_\perp^H(\bfI-\bfP_\bfV)\tilde{\bfu}-\bfU_\perp^H
(\bfI-\bfP_\bfV)\bfu\right\|}{\|(\bfI-\bfP_\bfV)\bfu\|}\\
\label{sin_tv_v}&=&\frac{\left\|\bfU_\perp^H\left((\bfI-\bfP_\bfV)
\tilde{\bfu}-(\bfI-\bfP_\bfV)\bfu\right)\right\|}{\|(\bfI-\bfP_\bfV)\bfu\|}.
\end{eqnarray}
From (\ref{IPtu}) we have $(\bfI-\bfP_\bfV)\tilde{\bfu}-(\bfI-\bfP_\bfV)
\bfu=\varepsilon\|\bfu\|\bff_\perp$. Substituting it into (\ref{sin_tv_v}) gives
$$
\sin\angle(\tilde{\bfv},\bfv)
=\tilde{\varepsilon}\sin\angle(\tilde{\bfv},\bff_\perp).
$$
\end{proof}


In order to make the inexact SIRA method mimic the SIRA method well,
we must require that $\tilde{\bfv}$ approximates $\bfv$
with certain accuracy, i.e., $\tilde\varepsilon$ suitably small,
so that the two expanded subspaces have comparable quality. We will come back
to this key point and estimate $\tilde\varepsilon$ quantitatively
in Section \ref{sec:teps}.

In what follows we establish an important relationship between $\varepsilon$ and
$\tilde{\varepsilon}$, and based on it we analyze how
$\varepsilon$ varies with $\alpha_1$ and $\alpha_2$ for a given $\tilde\varepsilon$.

\begin{Thm}\label{thm4}
Let $\bfy$ be the current approximate eigenvector
and $\alpha=-\frac{\alpha_2}{\alpha_1}$ with $\alpha_1,\alpha_2$
in {\rm (\ref{alpha12})}. We have
\begin{equation}
\label{eps_teps_relation_general}\varepsilon \leq \frac{2\|\bfB \|
\sin\angle(\bfy,\bfx)}{\left\|\bfB \bfy-\alpha \bfy\right\|\sin\angle(\mathcal{V},\bff)}
\tilde{\varepsilon}.
\end{equation}
\end{Thm}

\begin{proof}
By definition (\ref{fperp}), we have
$$
\|\bff_{\perp}\|=\|(\bfI-\bfP_\bfV)\bff\|=\sin\angle(\mathcal{V},\bff).
$$
From (\ref{IPtu}), we get
\begin{eqnarray*}
\varepsilon
&=&\frac{\|(\bfI-\bfP_\bfV)
\tilde{\bfu}-(\bfI-\bfP_\bfV)\bfu\|}{\|\bfu\|\|\bff_{\perp}\|}\\
&=&\frac{\|(\bfI-\bfP_\bfV)\bfu\|}{\|\bfu\|\|\bff_{\perp}\|}\frac{\|(\bfI-\bfP_\bfV)
\tilde{\bfu}-(\bfI-\bfP_\bfV)\bfu\|}{\|(\bfI-\bfP_\bfV)\bfu\|}
\\
&=&\frac{\|(\bfI-\bfP_\bfV)\bfu\|}{\|\bfu\|\|\bff_\perp\|}\tilde{\varepsilon}
=\frac{\|(\bfI-\bfP_\bfV)\bfu\|}{\|\bfu\|\sin\angle(\mathcal{V},\bff)}
\tilde{\varepsilon}.
\end{eqnarray*}
By (\ref{u_unified}), we substitute $\bfu=\alpha_1\bfB\bfy+\alpha_2\bfy$  into
the above, giving
\begin{eqnarray}
\varepsilon
&=&\frac{\|(\bfI-\bfP_\bfV)(\alpha_1\bfB\bfy+\alpha_2\bfy)\|}
{\|\alpha_1\bfB \bfy+\alpha_2\bfy\|\sin\angle(\mathcal{V},\bff)}
\tilde{\varepsilon} \nonumber\\
&=&\frac{\|\alpha_1(\bfI-\bfP_\bfV)\bfB\bfy\|}{\|\alpha_1\bfB
\bfy+\alpha_2\bfy\|\sin\angle(\mathcal{V},\bff)}\tilde{\varepsilon}  \nonumber\\
&=&\frac{\|(\bfI-\bfP_\bfV)\bfB \bfy\|}{\left\|\bfB \bfy+
\frac{\alpha_2}{\alpha_1}\bfy\right\|\sin\angle(\mathcal{V},\bff)}
\tilde{\varepsilon}. \label{eps_teps_relation}
\end{eqnarray}
Decompose $\bfy$ into the orthogonal direct sum
\begin{equation}
\label{decompose_y}\bfy=\cos\angle(\bfy,\bfx)\bfx+\sin\angle(\bfy,\bfx)\bfg
\end{equation}
with $\bfg\perp \bfx$ and $\|\bfg\|=1$. Then we get
\begin{eqnarray*}
(\bfI-\bfP_\bfV)\bfB \bfy
&=&(\bfI-\bfP_\bfV)\left(\cos\angle(\bfy,\bfx)\bfB \bfx+\sin\angle(\bfy,\bfx)\bfB
\bfg\right)\\
&=&(\bfI-\bfP_\bfV)\left(\frac{\cos\angle(\bfy,\bfx)}{\lambda-\sigma}\bfx+
\sin\angle(\bfy,\bfx)\bfB \bfg\right)\\
&=&\frac{\cos\angle(\bfy,\bfx)}{\lambda-\sigma}\xp+\sin\angle(\bfy,\bfx)
(\bfI-\bfP_\bfV)\bfB \bfg,
\end{eqnarray*}
where $\xp=(\bfI-\bfP_\bfV)\bfx$. Making use of $\|\xp\|=\sin\angle(\mathcal{V},\bfx)
\leq\sin\angle(\bfy,\bfx)$ and $\frac{1}{|\lambda-\sigma|}\leq\|\bfB\|$, we obtain
\begin{eqnarray}
\|(\bfI-\bfP_\bfV)\bfB \bfy\|
&=&\left\|\frac{\cos\angle(\bfy,\bfx)}{\lambda-\sigma}\xp+\sin\angle(\bfy,\bfx)
(\bfI-\bfP_\bfV)\bfB \bfg\right\|\nonumber\\
&\leq&\frac{|\cos\angle(\bfy,\bfx)|}{|\lambda-\sigma|}\|\xp\|+\|(\bfI-\bfP_\bfV)\bfB
\bfg\|\sin\angle(\bfy,\bfx)\nonumber\\
&\leq&\left(\frac{|\cos\angle(\bfy,\bfx)|}{|\lambda-\sigma|}+
\|(\bfI-\bfP_\bfV)\bfB\bfg\|\right)\sin\angle(\bfy,\bfx)\nonumber\\
&\leq&\left(\frac{1}{|\lambda-\sigma|}+\|\bfB \|\right)\sin\angle(\bfy,\bfx)
\nonumber\\
&\leq&2\|\bfB\|\sin\angle(\bfy,\bfx). \label{Bybound}
\end{eqnarray}
Therefore, combining the last relation with (\ref{eps_teps_relation}) establishes
(\ref{eps_teps_relation_general}).
\end{proof}

Observe that the linear system $(\bfA-\sigma\bfI)\bfu=\bfy$,
which is also the one in the inverse power method at each step,
falls into the form of (\ref{ls_unified}) by taking $\alpha_1=1$
and $\alpha_2=0$. For this case, from (\ref{eps_teps_relation_general})
we have
\begin{equation}
\label{eps_teps_relation_original}\varepsilon\leq\frac{2\|\bfB \|
\sin\angle(\bfy,\bfx)}{\left\|\bfB\bfy\right\|\sin\angle(\mathcal{V},\bff)}
\tilde{\varepsilon}.
\end{equation}
We comment that (i) $\sin\angle(\mathcal{V},\bff)$ is moderate
as $\bff$ is a general vector and (ii) $\|\bfB\|/\|\bfB\bfy\|=O(1)$
if $\bfy$ is a reasonably good approximation to $\bfx$
and in the worst case $\|\bfB\|/\|\bfB\bfy\|\leq\kappa(\bfB)$. In case that
$\sin\angle(\mathcal{V},\bff)$ is small, $\varepsilon$ becomes big for
a fixed small $\tilde{\varepsilon}$, that is, linear system (\ref{ls_unified})
is allowed to be solved with less accuracy. So a small
$\sin\angle(\mathcal{V},\bff)$ is a lucky event.

We can use this theorem to further illustrate why it is bad to solve
$(\bfA-\sigma\bfI)\bfu=\bfy$ iteratively. For a fixed small
$\tilde{\varepsilon}$, (\ref{eps_teps_relation_original}) tells us that
$\varepsilon$ should become smaller as $\sin\angle(\bfy,\bfx)\rightarrow 0$
as the algorithms converge.
As a result, we have to solve inner linear systems with higher accuracy as
$\bfy$ becomes more accurate.
More generally, this is the case when $\left\|\bfB \bfy-\alpha \bfy\right\|$
is not small and typically of $O(\|\bfB\|)$. Therefore, for $\alpha=0$ and more
general $\alpha$, the resulting method and SIA type methods are similar
and no winner in theory. They are common in that they all require to solve inner
linear systems accurately for some steps and they are different in that
the former solves inner linear systems with poor accuracy initially and then
with increasing accuracy as the algorithm converges, while
the latter ones solve inner linear systems with high accuracy in some initial
outer iterations and then with decreasing accuracy as the algorithms converge.

Based on (\ref{eps_teps_relation_general}), it is natural for us to maximize
its upper bound with respect to $\alpha$ for a
fixed $\tilde{\varepsilon}$. This will make $\varepsilon$ is
as small as possible, so that we pay least
computational efforts to solve (\ref{ls_unified}).
This amounts to minimizing $\left\|\bfB \bfy-\alpha \bfy\right\|$.
As is well known, the optimal $\alpha$ is
\begin{equation}
\arg\min\limits_{\alpha\in\mathcal{C}}\left\|\bfB \bfy-\alpha \bfy\right\|
=\bfy^H\bfB \bfy,
\end{equation}
Such $\alpha=-\frac{\alpha_2}{\alpha_1}$ corresponds to the choice
$\alpha_1=-\frac{1}{\bfy^H\bfB \bfy}$
and $\alpha_2=1$ in (\ref{ls_unified}), exactly leading to
linear system (\ref{ls_jd_1}) in the JD method. Therefore,
in the sense of minimizing $\left\|\bfB \bfy-\alpha \bfy\right\|$, the JD method
is the best. If we take $\alpha=\frac{1}{\nu-\sigma}$, which is the
approximation to $\frac{1}{\lambda-\sigma}$ in SIRA, by letting
$\alpha_1=\sigma-\nu$ and $\alpha_2=1$, then (\ref{ls_unified}) becomes
\begin{eqnarray*}
(\bfA-\sigma \bfI)\bfu=(\bfA-\sigma \bfI)\bfy+(\sigma-\nu)\bfy=\bfr_S,
\end{eqnarray*}
which is exactly the linear system in the SIRA method. In each of JD and SIRA,
$\left\|\bfB \bfy-\alpha \bfy\right\|$ is the residual norm of an
approximate eigenpair $(\alpha,\bfy)$ of $\bfB$.

In what follows, we denote $\varepsilon$ by $\varepsilon_S$ and $\varepsilon_J$
in the SIRA and JD methods, respectively. To derive our final and key relationships
between $\varepsilon_S,\,\varepsilon_J$ and $\tilde{\varepsilon}$, we
need the following lemma, which is direct from Theorem 6.1
of \cite{jia2001analysis} and establishes a close and compact relationship between
$\sin\angle(\bfy,\bfx)$ and the residual norm $\left\|\bfB \bfy-\alpha \bfy\right\|$.

\begin{Lem}\label{lem:upper_bound_sin_y_x}
Suppose $\left(\frac{1}{\lambda-\sigma},\bfx\right)$ is a simple
desired eigenpair of $\bfB\in\mathcal{C}^{n\times n}$ and
let $(\bfx,\bfX_\perp)$ be unitary. Then
\begin{equation}
\left[\begin{array}{c}\bfx^H \\ \bfX_\perp^H \end{array}\right]\bfB
\left[\begin{array}{cc}\bfx & \bfX_\perp \end{array}\right]=
\left[\begin{array}{cc}\frac{1}{\lambda-\sigma} & \bfc^H \\
\mathbf{0} & \bfL\end{array}\right],
\end{equation}
where $\bfc^H=\bfx^H\bfB\bfX_\perp$ and $\bfL=\bfX_\perp^H\bfB\bfX_\perp$.
Let $(\alpha,\bfy)$ be an approximation
to $\left(\frac{1}{\lambda-\sigma},\bfx\right)$, assume that $\alpha$
is not an eigenvalue of $\bfL$ and define
\begin{equation}
\sep\left(\alpha,\bfL\right)=\|(\bfL-\alpha\bfI)^{-1}\|^{-1}>0.
\end{equation}
Then
\begin{equation}
\label{upper_bound_sin_y_x}\sin\angle(\bfy,\bfx)\leq\frac{\|\bfB\bfy-
\alpha\bfy\|}{\sep\left(\alpha,\bfL\right)}.
\end{equation}
\end{Lem}

Combining (\ref{upper_bound_sin_y_x}) with Theorem~\ref{thm4}, we obtain one of our
main results.

\begin{Thm}\label{thm5}
Assume that $\alpha$ is an approximation to $\frac{1}{\lambda-\sigma}$
and is not an eigenvalue of $\bfL$. Then
\begin{equation}
\label{upper_bound_eps}\varepsilon\leq\frac{2\|\bfB \|}
{\sep\left(\alpha,\bfL\right)\sin\angle(\mathcal{V},\bff)}\tilde{\varepsilon}.
\end{equation}
In particular, for $\alpha=\frac{1}{\nu-\sigma}$ and
$\alpha=\bfy^H\bfB\bfy$, which correspond to
the SIRA and JD methods, respectively, assume that each of them is not an
eigenvalue of $\bfL$. Then it holds that
\begin{equation}
\label{eps_s}\varepsilon_S\leq\frac{2\|\bfB \|}{\sep\left(\frac{1}
{\nu-\sigma},\bfL\right)\sin\angle(\mathcal{V},\bff)}\tilde{\varepsilon},
\end{equation}
and
\begin{equation}
\label{eps_j}\varepsilon_J\leq\frac{2\|\bfB \|}
{\sep\left(\bfy^H\bfB\bfy,\bfL\right)\sin\angle(\mathcal{V},\bff)}
\tilde{\varepsilon}.
\end{equation}
\end{Thm}

This theorem shows that once $\tilde\varepsilon$ is known we
can a-priori determine the accuracy requirements $\varepsilon_S$
and $\varepsilon_J$ on approximate solutions of inner linear
systems (\ref{lssira}) and (\ref{ls_jd}).

It is important to observe from (\ref{upper_bound_eps}) that
\begin{eqnarray*}
\varepsilon
\leq\frac{2\|\bfB \|}{\sep\left(\alpha,\bfL\right)\sin\angle(\mathcal{V},\bff)}
\tilde{\varepsilon}=\frac{2\|\bfB \|}{O(\|\bfB\|)}\tilde{\varepsilon}
=O(\tilde{\varepsilon})
\end{eqnarray*}
if $\alpha$ is well separated from the eigenvalues of $\bfB$ other than
$\frac{1}{\lambda-\sigma}$ and $\bfB$ is normal or mildly non-normal
and $\sin\angle(\mathcal{V},\bff)$
is not small. For $\sin\angle(\mathcal{V},\bff)$ small, noting that bound
(\ref{upper_bound_eps}) is compact,
we are lucky to have a bigger $\varepsilon$, i.e., to solve the inner
linear system with less accuracy. If $\sep\left(\alpha,\bfL\right)$
is considerably smaller than $\|\bfB\|$, then $\varepsilon$
may be bigger than $\tilde{\varepsilon}$ considerably and we are
likely lucky to solve the inner linear system with less accuracy.

For the $\alpha$'s in the SIRA and JD methods, by continuity the corresponding
two $\sep\left(\alpha,\bfL\right)$'s are close. Therefore, for
a given $\tilde{\varepsilon}$,
we have essentially the same upper bounds for $\varepsilon_S$ and
$\varepsilon_J$. This means that we need to solve the corresponding
inner linear systems (\ref{lssira}) and (\ref{ls_jd}) in the SIRA and
JD methods with essentially the same accuracy $\varepsilon$.
In other words, the SIRA and JD methods behave very similar when
(\ref{lssira}) and (\ref{ls_jd}) are solved with the same accuracy.

\section{Subspace improvement and selection of
$\tilde{\varepsilon}$ and $\varepsilon$}\label{sec:teps}

In this section, we first focus on the fundamental problem of how to
select $\tilde{\varepsilon}$
to make the inexact SIRA and JD mimic the exact SIRA very well from
the current step to the next one. Then we show how to achieve our
ultimate goal: the determination of $\varepsilon$.

Recall that the subspace expansion vectors are $\bfv$ and
$\tilde{\bfv}$ for the exact SIRA and the inexact SIRA or JD; see
(\ref{tv_and_v}). Define $\bfV_+=\left[\begin{array}{cc}
\bfV & \bfv \end{array}\right]$,
$\mathcal{V}_+=span\left\{\bfV_+\right\}$ and $\tilde{\bfV}_+
=\left[\begin{array}{cc}\bfV&\tilde{\bfv}\end{array}\right]$,
$\tilde{\mathcal{V}}_+=span\{\tilde{\bfV}_+\}$.
In order to make the inexact SIRA method mimic the exact SIRA method very
well, we must require that the two expanded subspaces $\mathcal{V}_+$ and
$\tilde{\mathcal{V}}_+$ have almost the same quality, namely,
$\sin\angle(\tilde{\mathcal{V}}_+,\bfx)\approx\sin\angle(\mathcal{V}_+,\bfx)$,
whose quantitative meaning will be clear later.

\begin{Thm}\label{thm:sin_v_xp_div_sin_tv_xp}
With the notations above, assume
$\sin\angle(\bfv,\xp)\not=0$ with $\bfx_\perp=(\bfI-\bfP_\bfV)\bfx$.\footnote
{If it fails to hold,
it is seen from (\ref{sin_v_xp}) that $\sin\angle(\mathcal{V}_+,\bfx)=0$
and the exact SIRA, SIA and JD methods terminate prematurely if
$\dim(\mathcal{V}_+)<n$. In this case,
$\mathcal{V}_+$ is an invariant subspace of $\bfA$
and we stop subspace expansion.
We will exclude this rare case.}
Then we have
\begin{eqnarray}
\label{sin_v_xp} \sin\angle(\mathcal{V}_+,\bfx)&=&
\sin\angle(\mathcal{V},\bfx) \sin\angle(\bfv,\xp),\\
\label{sin_v_xp_div_sin_tv_xp}
\frac{\sin\angle(\tilde{\mathcal{V}}_+,\bfx)}{\sin\angle(\mathcal{V}_+,\bfx)}
&=&\frac{\sin\angle(\tilde{\bfv},\xp)}{\sin\angle(\bfv,\xp)}.
\end{eqnarray}
Suppose $\angle(\tilde\bfv,\bfv)$ is acute. If
$\tau=\frac{2\tilde{\varepsilon}}{\sin\angle(\bfv,\xp)}<1$, we have
\begin{equation}
1-\tau\leq \label{sin_tV+_x_div_sin_V+_x}
\frac{\sin\angle(\tilde{\mathcal{V}}_+,\bfx)}{\sin\angle(\mathcal{V}_+,\bfx)}
\leq 1+\tau. \label{tau}
\end{equation}
\end{Thm}

\begin{proof}
Since
\begin{eqnarray*}
\sin^2\angle(\mathcal{V},\bfx)-\sin^2\angle(\mathcal{V}_+,\bfx)=
\|(\bfI-\bfP_\bfV)\bfx\|^2-\|(\bfI-\bfP_{\bfV_+})\bfx\|^2=|\bfv^H\bfx|^2,
\end{eqnarray*}
by $\|\xp\|=\sin\angle(\mathcal{V},\bfx)$ we obtain
\begin{eqnarray*}
\frac{\sin\angle(\mathcal{V}_+,\bfx)}{\sin\angle(\mathcal{V},\bfx)}
\nonumber&=&\sqrt{1-\left(\frac{|\bfv^H\bfx|}{\sin\angle(\mathcal{V},\bfx)}\right)^2}\\
\nonumber&=&\sqrt{1-\left(\frac{|\bfv^H\xp|}{\sin\angle(\mathcal{V},\bfx)}\right)^2}\\
\nonumber&=&\sqrt{1-\left(\frac{\|\xp\|\cos\angle(\bfv,\xp)}
{\sin\angle(\mathcal{V},\bfx)}\right)^2}\\
\nonumber&=&\sqrt{1-\cos^2\angle(\bfv,\xp)}\\
&=&\sin\angle(\bfv,\xp),
\end{eqnarray*}
which proves (\ref{sin_v_xp}). Similarly, we have
\begin{equation}
\label{sin_tv_xp}\frac{\sin\angle(\tilde{\mathcal{V}}_+,\bfx)}
{\sin\angle(\mathcal{V},\bfx)}
=\sin\angle(\tilde{\bfv},\xp).
\end{equation}
Hence, from (\ref{sin_v_xp}) and (\ref{sin_tv_xp}), we get (\ref{sin_v_xp_div_sin_tv_xp}).

Exploiting the trigonometric identity
$$
\sin\angle(\tilde\bfv,\bfx_{\perp})-\sin\angle(\bfv,\bfx_{\perp})
=2\cos\frac{\angle(\tilde\bfv,\bfx_{\perp})+\angle(\bfv,\bfx_{\perp})}{2}
\sin\frac{\angle(\tilde\bfv,\bfx_{\perp})-\angle(\bfv,\bfx_{\perp})}{2},
$$
the angle triangle inequality
$$
|\angle(\tilde\bfv,\bfx_{\perp})-\angle(\bfv,\bfx_{\perp})|\leq\angle(\tilde{\bfv},\bfv).
$$
and the monotonic increasing property of the $\sin$ function in the first quadrant,
we get
\begin{eqnarray}
|\sin\angle(\tilde\bfv,\bfx_{\perp})-\sin\angle(\bfv,\bfx_{\perp})|
&\leq&2\left|\sin\frac{\angle(\tilde\bfv,\bfx_{\perp})-\angle(\bfv,\bfx_{\perp})}{2}\right|
\nonumber\\
&=&2\sin\frac{|\angle(\tilde\bfv,\bfx_{\perp})-\angle(\bfv,\bfx_{\perp})|}{2}\nonumber\\
&\leq&2\sin\frac{\angle(\tilde{\bfv},\bfv)}{2}\nonumber\\
&\leq&2\sin\angle(\tilde{\bfv},\bfv). \label{triangle}
\end{eqnarray}

From (\ref{sin_v_xp_div_sin_tv_xp}), (\ref{triangle}) and (\ref{sin_tv_v_final}),
we obtain
\begin{eqnarray*}
\left|\frac{\sin\angle(\tilde{\mathcal{V}}_+,\bfx)}
{\sin\angle(\mathcal{V}_+,\bfx)}-1\right|
&=&\left|\frac{\sin\angle(\tilde{\bfv},\xp)}
{\sin\angle(\bfv,\xp)}-1\right|\\
&=&\frac{\left|\sin\angle(\tilde{\bfv},\xp)-
\sin\angle(\bfv,\xp)\right|}{\sin\angle(\bfv,\xp)}\\
&\leq&\frac{2\sin\angle(\tilde{\bfv},\bfv)}
{\sin\angle(\bfv,\xp)}\\
&\leq&\frac{2\tilde{\varepsilon}}{\sin\angle(\bfv,\xp)}=\tau,
\end{eqnarray*}
from which it follows that (\ref{tau}) holds.
\end{proof}

From (\ref{sin_v_xp}), we see that $\sin\angle(\bfv,\xp)$
is exactly one step subspace improvement when
$\mathcal{V}$ is expanded to $\mathcal{V}_+$.

(\ref{tau}) shows that, to make $\sin\angle(\tilde{\mathcal{V}}_+,\bfx)
\approx\sin\angle(\mathcal{V}_+,\bfx)$, $\tau$ should
be small. Meanwhile, (\ref{tau}) also indicates that a very small $\tau$ cannot
improve the bounds essentially. Actually, for our purpose,
a fairly small $\tau$, e.g., $\tau=0.01$, is enough since we have
$$
0.99\leq\frac{\sin\angle(\tilde{\mathcal{V}}_+,\bfx)}
{\sin\angle(\mathcal{V}_+,\bfx)}\leq 1.01
$$
and the lower and upper bounds are very near and differ
marginally. Therefore, $\tilde{\mathcal{V}}_+$ and $\mathcal{V}_+$
are of almost the same quality for approximating $\bfx$.
As a result, it is expected that the inexact SIRA or JD computes new
approximation over $\tilde{\mathcal{V}}_+$ to the desired
$(\lambda,\bfx)$ that has almost the same accuracy as that obtained by
the exact SIRA over $\mathcal{V}_+$. More precisely, the accuracy of
the approximate eigenpair by the exact SIRA and that
by the inexact SIRA or JD are generally the same within roughly
a multiple $c\in [1-\tau,1+\tau]$
(this assertion can be justified from the results
in \cite{jia05,jia2001analysis}). So how near the constant $c$ is to one is
insignificant, the inexact SIRA and JD generally mimic the exact SIRA
very well when $\tau$ is fairly small. Concisely, we may well draw the conclusion that
$\tau=0.01$ makes the inexact SIRA mimic the exact
SIRA very well, that is, the exact and inexact SIRA methods use almost
the same outer iterations to achieve the convergence.

Next we discuss the selection of $\tilde\varepsilon$. Once
$\tilde\varepsilon$ is available, in principle we can exploit compact bounds
(\ref{eps_s}) and (\ref{eps_j}) to determine the accuracy
requirements $\varepsilon_S$ and $\varepsilon_J$ on inner iterations
in the SIRA and JD.

From the definition of $\tau$, we have
\begin{equation}
\label{teps}\tilde{\varepsilon}=\frac{\tau}{2}\sin\angle(\bfv,\xp).
\end{equation}
As Theorem~\ref{thm:sin_v_xp_div_sin_tv_xp} requires $\tau<1$,
we must have $\tilde{\varepsilon}<\frac{1}{2}\sin\angle(\bfv,\bfx_{\perp})$.
But $\xp$ is not available and a-priori, so we can only use a reasonable
estimate on $\sin\angle(\bfv,\xp)$ in (\ref{teps}). In the following,
we will look into $\sin\angle(\bfv,\xp)$
and show that it is actually independent of the quality of the approximate
eigenvector $\bfy$, i.e., $\sin\angle(\bfy,\bfx)$, and the subspace quality, i.e.,
$\sin\angle({\cal V},\bfx)$. This means that $\sin\angle(\bfv,\xp)$ stays around
some constant during outer iterations. Then we analyze its size, which is shown
to be problem dependent and stay around some certain constant during outer
iterations. Based on these results, we can propose a general practical selection
of $\tilde\varepsilon$. Obviously, in order to achieve a given $\tau$,
the smaller $\sin\angle(\bfv,\xp)$ is, the smaller $\tilde{\varepsilon}$
must be and the more accurately we need to solve the inner linear system.

We now investigate $|\cos\angle(\bfv,\bfx_\perp)|$ and show that
it is bounded independently of $\sin\angle(\bfy,\bfx)$ and
$\sin\angle({\cal V},\bfx)$, so is $\sin\angle(\bfv,\bfx_\perp)$.
From (\ref{expand}) and (\ref{tv_and_v}), it is known
that $\bfv$ and $(\bfI-\bfP_\bfV)\bfB\bfy$
are in the same direction. Therefore, from decomposition (\ref{decompose_y})
of $\bfy$, we have
\begin{eqnarray*}
|\cos\angle(\bfv,\bfx_\perp)|
&=&\frac{|\bfx_\perp^H(\bfI-\bfP_\bfV)\bfB \bfy|}
{\|\bfx_\perp\|\|(\bfI-\bfP_\bfV)\bfB \bfy\|}\\
&=&\frac{\left|\bfx_\perp^H(\bfI-\bfP_\bfV)\bfB
(\cos\angle(\bfy,\bfx)\bfx+\sin\angle(\bfy,\bfx)\bfg)\right|}
{\|\bfx_\perp\|\|(\bfI-\bfP_\bfV)\bfB \bfy\|}\\
&=&\frac{\left|\bfx_\perp^H(\bfI-\bfP_\bfV)\left(\frac{\cos\angle(\bfy,\bfx)}
{\lambda-\sigma}\bfx+\sin\angle(\bfy,\bfx)\bfB \bfg\right)\right|}
{\|\bfx_\perp\|\|(\bfI-\bfP_\bfV)\bfB \bfy\|}\\
&=&\frac{\left|\cos\angle(\bfy,\bfx)\|\bfx_\perp\|^2+(\lambda-\sigma)
\sin\angle(\bfy,\bfx)\bfx_\perp^H\bfB \bfg\right|}{|\lambda-\sigma|\|
\bfx_\perp\|\|(\bfI-\bfP_\bfV)\bfB \bfy\|}\\
&\leq&\frac{|\cos\angle(\bfy,\bfx)|\|\bfx_\perp\|}{|\lambda-\sigma|
\|(\bfI-\bfP_\bfV)\bfB \bfy\|}+\frac{\sin\angle(\bfy,\bfx)|
\bfx_\perp^H\bfB \bfg|}{\|\bfx_\perp\|\|(\bfI-\bfP_\bfV)\bfB \bfy\|}.
\end{eqnarray*}
Note that $|\bfx_\perp^H\bfB \bfg|\leq\|\bfx_\perp\|\|\bfB \bfg\|
\leq\|\bfx_\perp\|\|\bfB\|$ and
$\|\bfx_\perp\|=\sin\angle(\mathcal{V},\bfx)\leq\sin\angle(\bfy,\bfx)$.
So
\begin{eqnarray}
|\cos\angle(\bfv,\bfx_\perp)|
\nonumber&\leq&\frac{|\cos\angle(\bfy,\bfx)|\|\xp\|}{|\lambda-\sigma|
\|(\bfI-\bfP_\bfV)\bfB \bfy\|}+\frac{\sin\angle(\bfy,\bfx)
\|\bfB \bfg\|}{\|(\bfI-\bfP_\bfV)\bfB \bfy\|}\\
\nonumber&\leq&\left(\frac{|\cos\angle(\bfy,\bfx)|}
{|\lambda-\sigma|}+\|\bfB\|\right)\frac{\sin\angle(\bfy,\bfx)}
{\|(\bfI-\bfP_\bfV)\bfB \bfy\|}\\
\label{cosvxp}&\leq&\frac{2\|\bfB\|\sin\angle(\bfy,\bfx)}
{\|(\bfI-\bfP_\bfV)\bfB \bfy\|}.
\end{eqnarray}
Combining (\ref{cosvxp}) and (\ref{Bybound}), we have
\begin{eqnarray}
|\cos\angle(\bfv,\bfx_\perp)|\leq\frac{O(\|\bfB\|)\sin\angle(\bfy,\bfx)}
{O\left(\|\bfB\|\right)\sin\angle(\bfy,\bfx)}=O(1),
\end{eqnarray}
a seemingly trivial bound. However, the proof clearly shows that
our derivation is general and does not miss anything essential.
We are not able to make the bound essentially sharper and more elegant
as the inequalities used in the proof cannot be sharpened generally.
Nevertheless, this is enough for our purpose. A key implication is that
the bound is independent of $\sin\angle(\bfy,\bfx)$ and
$\sin\angle({\cal V},\bfx)$, so $|\cos\angle(\bfv,\bfx_\perp)|$ is
expected to be around some constant during outer iterations,
so is $\sin\angle(\bfv,\bfx_\perp)$.

It is possible to estimate $\sin\angle(\bfv,\xp)$ in some important
cases. For the starting vector $\bfv_1$, it is known that the exact SIRA, SIA
and JD methods work on the standard Krylov subspaces $\mathcal{V}=\mathcal{V}_m=
\mathcal{K}_m(\bfB,\bfv_1)$ and
$\mathcal{V}_+=\mathcal{V}_{m+1}=\mathcal{K}_{m+1}(\bfB,\bfv_1)$.
Here we have temporarily added iteration subscripts and assume that the current
iteration step is $m$. It is direct from (\ref{sin_v_xp_div_sin_tv_xp}) to get
\begin{equation}\label{sinprod}
\sin\angle(\mathcal{V}_{m+1},\bfx)=\sin\angle(\bfv_1,\bfx)\prod_{i=2}^{m+1}
\sin\angle(\bfv_i,\bfx_{i,\perp}),
\end{equation}
where the $\bfv_i$ are exact subspace expansion vectors and 
$\bfx_{i,\perp}=(\bfI-\bfP_{{\bfV_i}})\bfx$at steps $i=2,3,\ldots,m+1$.

For the Krylov subspaces ${\cal V}_m$ and ${\cal V}_{m+1}$, there have
been some estimates on $\sin\angle(\mathcal{V}_{m+1},\bfx)$
in \cite{jia95,jia98,saad1992eigenvalue}.  For $\bfB$ is diagonalizable,
suppose all the $\lambda_i,\ i=1,2,\ldots,n$ and $\sigma$ are real and
$\frac{1}{\lambda-\sigma}$ is also the algebraically largest
eigenvalue of $\bfB$, and define
$$
\eta=1+2\frac{\frac{1}{\lambda-\sigma}-\frac{1}{\lambda_2-\sigma}}
{\frac{1}{\lambda_2-\sigma}-\frac{1}{\lambda_n-\sigma}}
=1+2\frac{(\lambda_2-\lambda)(\lambda_n-\sigma)}
{(\lambda_n-\lambda_2)(\lambda-\sigma)}>1.
$$
Then it is shown in \cite{jia98,saad1992eigenvalue} that
$$
\sin\angle(\mathcal{V}_{m+1},\bfx)=\sin\angle(\bfv_1,\bfx)\prod_{i=2}^{m+1}
\sin\angle(\bfv_i,\bfx_{i,\perp})\leq C_{\bfv_1}\sin\angle(\bfv_1,\bfx)\left(\frac{1}
{\eta+\sqrt{\eta^2-1}}\right)^m,
$$
where $C_{\bfv_1}$ is a certain constant only depending on $\bfv_1$ and the conditioning
of the eigensystem of $\bfB$. So, ignoring the constant factor $C_{\bfv_1}$, we see
the product $\prod_{i=2}^{m+1}
\sin\angle(\bfv_i,\bfx_{i,\perp})$ converges to zero at least as rapidly as
$$
\left(\frac{1}
{\eta+\sqrt{\eta^2-1}}\right)^m.
$$
As we have argued, all the $\sin\angle(\bfv_i,\bfx_{i,\perp})$,
$i=2,3,\ldots,m+1$, stay around a certain constant.
So basically, each step subspace improvement $\sin\angle(\bfv_i,\bfx_{i,\perp}),
\ i=2,3,\ldots,m+1$, behaves like and is no more than the factor
$$
\frac{1}{\eta+\sqrt{\eta^2-1}},
$$
the average convergence factor for one step. Returning to our notation, we see
the size of $\sin\angle(\bfv,\xp)$ crucially
depends on the eigenvalue distribution. The better $\frac{1}{\lambda-\sigma}$ is
separated from the other eigenvalues of $\bfB$, the smaller
$\sin\angle(\bfv,\xp)$ is. Conversely,
if $\frac{1}{\lambda-\sigma}$ is poorly separated from the others,
$\sin\angle(\bfv,\xp)$ may be near to one. For more complicated complex
eigenvalues and/or $\sigma$, quantitative results are obtained for
$\sin\angle(\mathcal{V}_{m+1},\bfx)$ and
similar conclusions are drawn in \cite{jia95,jia98}.
However, we should point that these estimates
may be conservative and also only predict linear convergence. In practice,
a slightly superlinear convergence may occur sometimes, as has been observed in
\cite{leestewart07}.

For $\tau=0.01$, if $\sin\angle(\bfv,\xp)\in [0.02,0.2]$, then by
(\ref{teps}) we have $\tilde{\varepsilon}\in [10^{-4}, 10^{-3}]$.
Such $\sin\angle(\bfv,\xp)$ means that $\frac{1}{\lambda-\sigma}$ is well
separated from the other eigenvalues of $\bfB$ and the exact SIRA generally
converges fast. In practice, however, for a
given $\tilde{\varepsilon}$ we do not know the value of $\tau$ produced
by $\tilde{\varepsilon}$ as $\sin\angle(\bfv,\bfx_{\perp})$ and
its bound are not known.
For a given $\tilde{\varepsilon}$, if we are unlucky to get a $\tau$
not small like $0.01$, the inexact SIRA may use more outer iterations
than the exact SIRA.
Suppose we select $\tilde\varepsilon=\frac{10^{-3}}{2}$. Then if each
$\sin\angle(\bfv,\bfx_{\perp})=0.1$, we get
$\tau=0.01$. For this case,
we have a very good subspace $\mathcal{V}_m$ for $m=10$ since
$\sin(\mathcal{V}_{10},\bfx)\leq 10^{-9}$, so the exact
SIRA generally converges very fast! For a real-world problem, however,
one should not expect that $\frac{1}{\lambda-\sigma}$ is generally so well
separated from the other eigenvalues that the convergence can
be so rapid. Therefore, we generally expect that
$\tilde\varepsilon\in [10^{-4},10^{-3}]$ makes $\tau\leq 0.01$, so that
the inexact SIRA and JD mimic the exact SIRA very well.

Summarizing the above, we propose taking
\begin{equation}
\tilde{\varepsilon}\in [10^{-4}, 10^{-3}].\label{tildeepsilon}
\end{equation}

Our ultimate goal is to determine $\varepsilon_S$ and $\varepsilon_J$ for
the inexact SIRA and JD. Compact bounds (\ref{eps_s}) and (\ref{eps_j}) show
that they are generally of $O(\tilde{\varepsilon})$. However, it
is impossible to compute the bounds cheaply and accurately.
We will consider their practical estimates on $\varepsilon_S$ and $\varepsilon_J$
in Section~\ref{issue}, where
we demonstrate that these estimates are cheaply obtainable.


\section{Restarted algorithms and practical stopping criteria for inner
iterations}\label{issue}

Due to the storage requirement and computational cost,
Algorithms~\ref{alg:sira}--\ref{alg:jd} will be impractical for large steps
of outer iterations. To be practical, it is necessary to restart them
for difficult problems.
Let $\bf\bfM_{\max}$ be the maximum of outer iterations allowed. If the
basic SIRA and JD algorithms do not converge, then we simply update $\bfv_1$
and restart them. We call the resulting restarted algorithms
Algorithms~3--4, respectively.

In implementations, we adopt the following strategy to update $\bfv_1$.
For outer iteration steps $i=1,2,
\ldots,\bfM_{\max}$ during the current cycle, suppose $(\nu_1^{(i)},\bfy_1^{(i)})$ is
the candidate for approximating the desired eigenpair $(\lambda,x)$ of $\bfA$ at
the $i$-th outer iteration. Then we take
\begin{equation}\label{revector}
\bfv_1=\bfy=\arg\min_{i=1,2,\ldots,\bfM_{\max}}
\|(\bfA-\nu_1^{(i)} \bfI)\bfy_1^{(i)}\|
\end{equation}
as the updated starting vector in the next cycle.
Such a restarting strategy guarantees that
we use the {\em best} candidate Ritz vector in the sense
of (\ref{revector}) to restart the algorithms.

In what follows we consider some practical issues and design practical stopping
criteria for inner iterations in the (non-restarted and restarted) inexact SIRA
and JD algorithms.

Given $\tilde{\varepsilon}$, since $\bfL$ is not available, it is impossible
to compute ${\rm sep}(\frac{1}{\nu-\sigma},\bfL)$ and ${\rm sep}(\bfy^H\bfB\bfy,\bfL)$
in (\ref{eps_s}) and (\ref{eps_j}). Also, we cannot compute $\sin\angle(\cal V,\bff)$
in (\ref{eps_s}) and (\ref{eps_j}).
In practice, we simply replace the insignificant factor $\sin\angle(\cal V,\bff)$
by one, which makes $\varepsilon_S$ and $\varepsilon_J$ as small as possible, so that
the inexact SIRA and JD algorithms are the safest to mimic the exact SIRA.
We replace $\|\bfB\|$ by $\frac{1}{|\nu-\sigma|}$ in the inexact
SIRA and JD, respectively. For ${\rm sep}(\frac{1}{\nu-\sigma},\bfL)$,
we can exploit the spectrum information of $\bfH$ to estimate it. Let
$\nu_i,\,i=2,3,\ldots,m$ be the other eigenvalues (Ritz values) of $\bfH$
other than $\nu$. Then we use the estimate
\begin{equation}\label{sepL}
{\rm sep}\left(\frac{1}{\nu-\sigma},\bfL\right)\approx \min_{i=2,3,\ldots,m}
\left|\frac{1}{\nu-\sigma}-\frac{1}{\nu_i-\sigma}\right|.
\end{equation}
Note that it is very expensive to compute $\bfy^H\bfB\bfy$ but
$\bfy^H\bfB\bfy\approx\frac{1}{\nu-\sigma}$. So we simply use
$\frac{1}{\nu-\sigma}$ to estimate ${\rm sep}\left(\bfy^H\bfB\bfy,\bfL\right)$.
With these estimates and taking the equalities in compact bounds
(\ref{eps_s}) and (\ref{eps_j}), we get
\begin{equation}\label{epsilon}
\varepsilon_S=\varepsilon_J=\varepsilon
= 2\tilde{\varepsilon}\max\limits_{i=2,3,\ldots,m}
\left|\frac{\nu_i-\sigma}{\nu_i-\nu}\right|.
\end{equation}
It might be possible to have $\varepsilon\geq 1$ for a given
$\tilde\varepsilon$. This would make $\tilde\bfu$ no accuracy
as an approximation to $\bfu$. As a remedy, from now on we set
\begin{equation}\label{epsmin}
\varepsilon=\min\{\varepsilon,0.1\}.
\end{equation}
For $m=1$, we simply set $\varepsilon=\tilde\varepsilon$.

Note that $\frac{\|\tilde{\bfu}-\bfu\|}{\|\bfu\|}$ is a-priori and
uncomputable. We are not able to determine whether it is below
$\varepsilon$ or not. However, it is easy to verify that
\begin{equation}
\frac{1}{\kappa(\bfB)}\frac{\|\tilde{\bfu}-\bfu\|}{\|\bfu\|}\leq
\frac{\|\bfr_S-(\bfA-\sigma \bfI)\tilde{\bfu}\|}{\|\bfr_S\|}\leq\kappa(\bfB)
\frac{\|\tilde{\bfu}-\bfu\|}{\|\bfu\|} \label{prioris}
\end{equation}
and
\begin{equation}
\frac{1}{\kappa(\bfB')}\frac{\|\tilde{\bfu}-\bfu\|}{\|\bfu\|}\leq
\frac{\|-\bfr_J-(\bfI-\bfy\bfy^H)(\bfA-\sigma \bfI)(\bfI-\bfy\bfy^H)
\tilde{\bfu}\|}{\|\bfr_J\|}\leq\kappa(\bfB')
\frac{\|\tilde{\bfu}-\bfu\|}{\|\bfu\|}, \label{priorij}
\end{equation}
where $\tilde\bfu\perp\bfy$ and
$\bfB'=\bfB|_{\bfy^{\perp}}=(\bfA-\sigma\bfI)^{-1}|_{\bfy^{\perp}}$,
the restriction of $\bfB$ to the orthogonal complement of $span\{\bfy\}$.
Alternatively, based on the above two relations,
in practice we require that inner solves stop when the a-posteriori computable
relative residual norms
\begin{equation}
\label{stopcrit}\frac{\|\bfr_S-(\bfA-\sigma \bfI)\tilde{\bfu}\|}{\|\bfr_S\|}
\leq\varepsilon
\end{equation}
and
\begin{equation}
\label{stopcritjd}\frac{\|-\bfr_J-(\bfI-\bfy\bfy^H)(\bfA-\sigma \bfI)(\bfI-\bfy\bfy^H)
\tilde{\bfu}\|}{\|\bfr_J\|}
\leq\varepsilon
\end{equation}
for the inexact SIRA and JD, respectively.
\smallskip

{\em Remark.} In
\cite{spence2009ia,simoncini2005ia,simoncini2003ia},
a-priori accuracy requirements have been determined for inner iterations
in SIA type methods.
In computation,  a-posteriori residuals are intuitive,
and are probably the only practical way to approximate the a-priori residuals.
Here, by the above lower and upper bounds (\ref{prioris}) and (\ref{priorij})
that relate the a-posteriori relative residuals to the a-priori errors of
approximate solutions, we have simply demonstrated that
(\ref{stopcrit}) and (\ref{stopcritjd}) are reasonable stopping criteria
for inner solves.
We see that the a-priori errors and the a-posteriori errors
are definitely comparable once the linear systems are not ill conditioned.

\section{Numerical experiments}\label{numer}

We report numerical experiments to confirm our theory.
Our aims are mainly three-fold: (i) Regarding outer iterations, for
fairly small $\tilde\varepsilon=10^{-3}$ and $10^{-4}$,
the (non-restarted and restarted) inexact SIRA and JD
behave very like the (non-restarted and restarted) exact SIRA.
Even a bigger $\tilde\varepsilon=10^{-2}$ often works very well.
(ii) Regarding inner iterations and overall efficiency,
the inexact SIRA and JD algorithms are considerably more efficient
than the inexact SIA. (iii) SIRA and JD are similarly effective.

All the numerical experiments were performed on an Intel
(R) Core (TM)2 Quad CPU Q9400 $2.66$GHz with main memory 2 GB using
Matlab 7.8.0 with the machine precision $\epsilon_{\rm mach}=2.22\times
10^{-16}$ under the Microsoft Windows XP operating system.

At the $m$th step of the inexact SIRA or JD method,
we have $\bfH_m=\bfV_m^H\bfA\bfV_m$. Let $(\nu_i^{(m)},\bfz_i^{(m)}),\ i=1,2,\ldots,m$
be the eigenpairs of $\bfH_m$, which are ordered as
$$
|\nu_1^{(m)}-\sigma|<|\nu_2^{(m)}-\sigma|\leq\cdots\leq |\nu_m^{(m)}-\sigma|.
$$
We use the Ritz pair $(\nu_m,\bfy_m):=(\nu_1^{(m)},\bfV_m\bfz_1^{(m)})$
to approximate the desired eigenpair $(\lambda,x)$ of $\bfA$,
and the associated residual is $\bfr_m=\bfA\bfy_m-\nu_m\bfy_m$.

We stop the algorithms if
$$
\|\bfr_m\|\leq tol=\max\left\{\|\bfA\|_1,1\right\}\times10^{-10}.
$$
In the inexact SIRA and JD, we stop inner solves when (\ref{stopcrit}) and
(\ref{stopcritjd}) are satisfied, respectively,
and denote by SIRA($\tilde\varepsilon$) and JD($\tilde\varepsilon$)
the inexact SIRA and JD algorithms with the given parameter $\tilde\varepsilon$.
We use the following stopping criteria for inner iterations
in the exact SIRA and SIA algorithms and the inexact SIA algorithm.
\begin{itemize}
\item For the \textquotedblleft exact\textquotedblright\ SIRA algorithm,
we require the approximate solution $\tilde{\bfu}_{m+1}$ to satisfy
$$
\frac{\|\bfr_m-(\bfA-\sigma\bfI)\tilde{\bfu}_{m+1}\|}{\|\bfr_m\|}\leq10^{-14}.
$$
\item For the inexact SIA algorithm, we take the same outer iteration tolerance
$tol=\max\left\{\|\bfA\|_1,1\right\}\times10^{-10}$,
and use the stopping criterion (3.14) in \cite{spence2009ia} for inner solve,
where $\varepsilon=tol$
and the steps $m$ suitably bigger than the number
of outer iterations used by the exact SIRA so as to ensure the
convergence of the inexact SIA with the same accuracy. For the restarted
inexact SIA, we take $m$ the maximum outer iterations $\bfM_{\max}$ allowed
for each cycle.
\end{itemize}

In the numerical experiments, we always take the zero vector as
an initial approximate solution to each inner linear system and solve it
by the right-preconditioned GMRES(30) method. Outer iterations
start with the normalized vector $\frac{1}{\sqrt{n}}(1,1,\ldots,1)^H$.
For the correction equation in the JD method, we use
$$
\tilde{\bfM}_m=(\bfI-\bfy_m\bfy_m^H)\bfM(\bfI-\bfy_m\bfy_m^H),
$$
the restriction of $M$ to the orthogonal complement of $span\{\bfy_m\}$,
as a preconditioner, which is suggested in \cite{vandervorst2002eigenvalue}.
$\tilde{\bfM}_m^{-1}|_{\bfy_m^{\perp}}$
means the inverse of $\tilde{\bfM}_m$ restricted to the orthogonal
complement of $span\{\bfy_m\}$.
Here $\bfM\approx\bfA-\sigma\bfI$ is some preconditioner used for all the
inner linear systems involved in the algorithms tested except JD.
We use the Matlab function $[L,U]=ilu(A-sigma*speye(n),setup)$
to compute the sparse incomplete LU factorization of $A-\sigma I$
with a given dropping tolerance $setup.droptol$. We then take $M=LU$.
van der Vorst \cite{vandervorst2002eigenvalue} shows
how to use $\tilde{\bfM}_m$ as a left preconditioner for (\ref{ls_jd}).
It can also be used a right preconditioner for (\ref{ls_jd}) in the same spirit.
Adapted from \cite[p. 137-8]{vandervorst2002eigenvalue}, we briefly describe
how to do so. Suppose that a Krylov solver for (\ref{ls_jd}) with
right-preconditioning starts with zero vector
as an initial guess to the solution. Then the starting vector for the Krylov solver
is $\bfr_m$, which is in
the subspace orthogonal to $\bfy_m$, and all iteration vectors for the
Krylov solver are in that subspace. We compute
$\tilde{\bfM}_m^{-1}|_{\bfy_m^{\perp}}\bfw$ for a vector $\bfw$ supplied
by the Krylov solver at each inner iteration.
Let $\bfz=\tilde{\bfM}_m^{-1}|_{\bfy_m^{\perp}}\bfw$ and note that $\bfz\perp\bfy_m$.
Then it follows that
$$
\bfw=\tilde{\bfM}_m\bfz=(\bfI-\bfy_m\bfy_m^H)\bfM\bfz=\bfM\bfz-\beta\bfy_m,
$$
where $\beta=\bfy_m^H\bfM\bfz$. Equivalently,
$\bfz=\bfM^{-1}\bfw+\beta\bfM^{-1}\bfy_m$.
Again, using $\bfz\perp\bfy_m$, we have
$\bfy_m^H\bfM^{-1}\bfw+\beta\bfy_m^H\bfM^{-1}\bfy_m=0$,
i.e., $\beta=-\frac{\bfy_m^H\bfM^{-1}\bfw}{\bfy_m^H\bfM^{-1}\bfy_m}$.
Therefore, we can compute $\tilde{\bfM}_m^{-1}|_{\bfy_m^{\perp}}\bfw$ by
$$
\tilde{\bfM}_m^{-1}\bfw
=\bfM^{-1}\bfw-\left(\frac{\bfy_m^H\bfM^{-1}\bfw}{\bfy_m^H\bfM^{-1}\bfy_m}\right)
\bfM^{-1}\bfy_m.
$$

In all the tables below, we denote by $I_{out}$ the number of outer iterations
to achieve the convergence, by $I_{inn}$ the total number of inner iterations,
i.e., the products of the matrix $A$ by vectors used by the Krylov solver,
by $I_{0.1}$ the times of $\varepsilon=0.1$,
by $T_1$ the total CPU time of solving the small eigenproblems,
by $T_2$ the total CPU time of generating the orthonormal basis $\bfV$ and forming
the projection matrix $\bfH$, by $T_3$ the time of constructing the preconditioner
and by $T_4$ the total CPU time of the Krylov solver for solving right-preconditioned
inner linear systems. We point out that the (inexact and exact) SIRA and JD methods
must form the projection matrices explicitly while SIA does not and it gives
its projection matrix as a byproduct when generating the orthonormal basis of
$\bfV$. As a result, for the same dimension of subspace,
$T_2$ for SIA is smaller than that for SIRA and JD. This will be
confirmed clearly in later numerical experiments, and we will not mention
this observation later.
For Examples 1--3 we test Algorithms~\ref{alg:sira}--\ref{alg:jd},
the inexact SIA and exact SIRA; for Example 4 we test these algorithms and
the restarted Algorithms~3--4 as well as the restarted inexact SIA.

\smallskip

\noindent\textbf{Example 1.} This problem is a large nonsymmetric standard
eigenvalue problem of cry10000 of $n=10000$ that arises from the stability
analysis of a crystal growth problem from \cite{matrixmarket}.
We are interested in the eigenvalue nearest to $\sigma=7$.
The computed eigenvalue is $\lambda\approx6.7741$.
The preconditioner $\bfM$ is obtained by the sparse
incomplete LU factorization of $\bfA-\sigma \bfI$ with $setup.droptol=0.001$.
Table~\ref{tab_cry10000} reports the results obtained, and the left
and right parts of Figure~\ref{fig_cry10000} depict the convergence curve
of $\frac{\|\bfr_m\|}{\|\bfA\|_1}$ versus $I_{out}$ and the curve of $I_{inn}$
versus $I_{out}$ for the algorithms, respectively.

\begin{figure}[!htb]
\begin{minipage}{7.5cm}
\includegraphics[height=5.8cm]{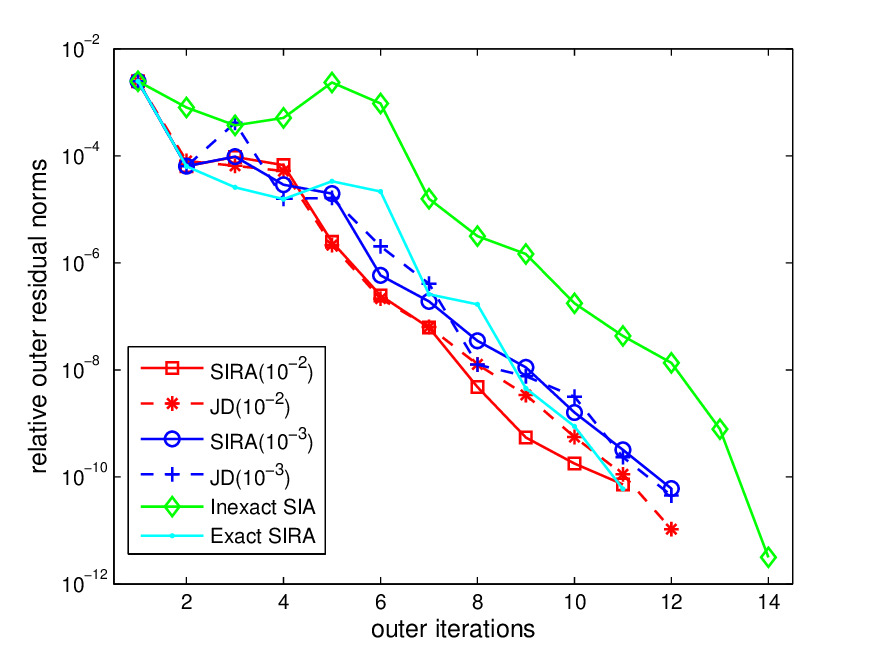}
\end{minipage}
\begin{minipage}{7.5cm}
\includegraphics[height=5.8cm]{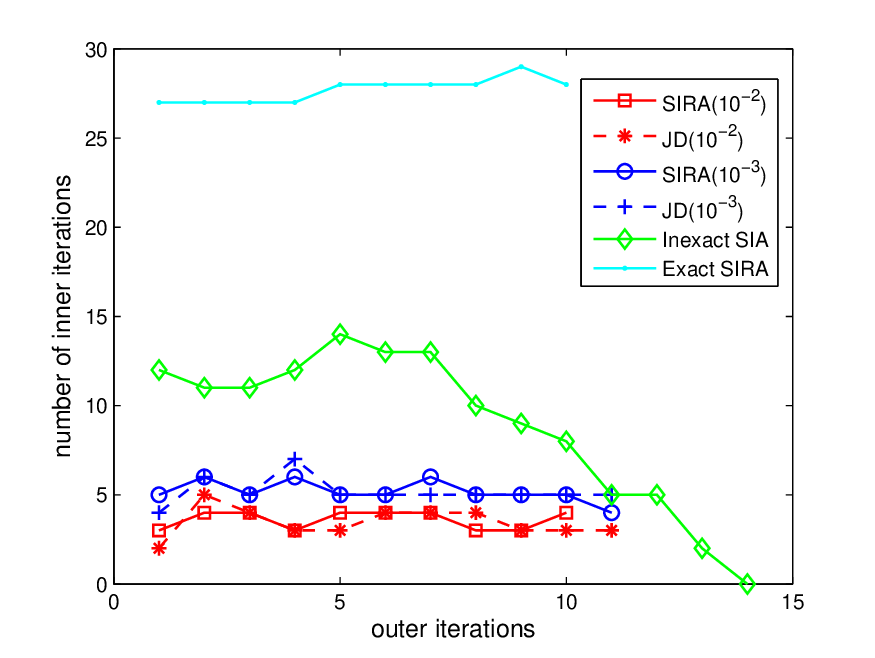}
\end{minipage}
\caption{\emph{Example 1. cry10000 with $\sigma=7$.
Left: relative outer residual norms versus outer iterations.
Right: the numbers of inner iterations versus outer iterations.}}
\label{fig_cry10000}
\end{figure}

\begin{table}[!htb]
\begin{center}
\begin{tabular}{rrrrrrrrrr}
\hline
Algorithm && $I_{inn}$ & $I_{out}$ & $I_{0.1}$ && $T_1$ & $T_2$ & $T_3$ & $T_4$ \\
\hline\hline
SIRA($10^{-2}$) && $36$  & $11$ & $0$ && $1$ & $18$ & $121$ & $81$ \\
JD($10^{-2}$)   && $38$  & $12$ & $0$ && $2$ & $21$ & $121$ & $103$\\ \hline
SIRA($10^{-3}$) && $57$  & $12$ & $0$ && $2$ & $21$ & $121$ & $120$ \\
JD($10^{-3}$)   && $57$  & $12$ & $0$ && $2$ & $21$ & $121$ & $136$\\ \hline
SIRA($10^{-4}$) && $88$  & $13$ & $0$ && $2$ & $24$ & $121$ & $184$\\
JD($10^{-4}$)   && $78$  & $12$ & $0$ && $2$ & $21$ & $121$ & $176$\\ \hline
Inexact SIA     && $131$ & $14$ & $-$ && $2$ & $13$ & $121$ & $340$\\
\textquotedblleft Exact\textquotedblright\ SIRA
                && $277$ & $11$ & $-$ && $1$ & $18$ & $121$ & $1386$\\ \hline \hline
\end{tabular}
\caption{\emph{Example 1. cry10000 with $\sigma=7$ (The unit of $T_1\sim T_4$ is
$0.001$ second).}}
\label{tab_cry10000}
\end{center}
\end{table}

We see from Table~\ref{tab_cry10000} and Figure~\ref{fig_cry10000}
that for both $\tilde{\varepsilon}=10^{-2},10^{-3}$ the inexact SIRA and JD
behaved like the exact SIRA very much and used almost the same outer
iterations, while the inexact SIA had a small convergence delay.
Clearly, smaller $\tilde{\varepsilon}$ is not necessary as
it cannot reduce outer iterations anymore.

Regarding the overall efficiency, the exact SIRA was obviously the most
expensive, as $I_{inn}$ and the dominant CPU time $T_3,\ T_4$
indicated. It used $27\sim29$ inner iterations per outer iteration.
The inexact SIA was the second most expensive, in terms of the same measures.
For it, the numbers of inner iterations were comparable and between $11\sim14$
at each of the first $7$ outer iterations where the accuracy of approximate
eigenpairs was poor and the inner linear systems must be solved with high accuracy.
As the approximate eigenpairs started converging, the relaxation strategy
came into picture and the inner linear systems were solved with decreasing
accuracy, leading to fewer inner iterations at subsequent outer iterations.
Inner iterations used by the inexact SIA were only comparable to and finally
below those used by the inexact SIRA and JD in the last very few iterations.
In contrast, the figure indicates that, for the same $\tilde{\varepsilon}$,
the inexact SIRA and JD solved the linear systems with almost the same inner
iterations per outer iteration.
Because of this, the inexact SIRA and JD were much more efficient than
the inexact SIA and used much fewer inner iterations and computing time
than the latter. Both the $I_{inn}$ and the total computing time in
Table~\ref{tab_cry10000} show that they were roughly one and a half to three
times as fast as the inexact SIA, and SIRA and JD with $\tilde{\varepsilon}=10^{-2}$
were considerably more efficient than that with $\tilde{\varepsilon}=10^{-3}$,
$10^{-4}$. Finally, we observe that the inexact SIRA and JD were equally effective,
as indicated by the $I_{inn}$ and the computing time
used for each $\tilde{\varepsilon}$.

In addition, we see from Table~\ref{tab_cry10000} that
$T_3$ is comparable to and can be more than $T_4$ when inner linear systems
are solved with low accuracy, and it is less important for the
inexact SIA, where the accuracy of inner
inner iterations increases as outer iterations proceed, and especially
for the exact SIRA, where inner linear systems are required to
be solved exactly in finite precision
arithmetic.

\smallskip

\noindent\textbf{Example 2.} We consider the unsymmetric sparse matrix sherman5
of $n=3312$ that has been used in \cite{spence2009ia,simoncini2005ia}
for testing the relaxation theory with $\sigma=0$. The computed eigenvalues
is $\lambda\approx4.6925\times10^{-2}$. The preconditioner $\bfM$ is obtained
by the sparse incomplete LU factorization of $\bfA-\sigma \bfI$ with
$setup.droptol=0.001$.
Table~\ref{tab_cry10000} and Figure~\ref{fig_cry10000} describe
the results and convergence processes.

\begin{figure}[!htb]
\begin{minipage}{7.5cm}
\includegraphics[height=5.8cm]{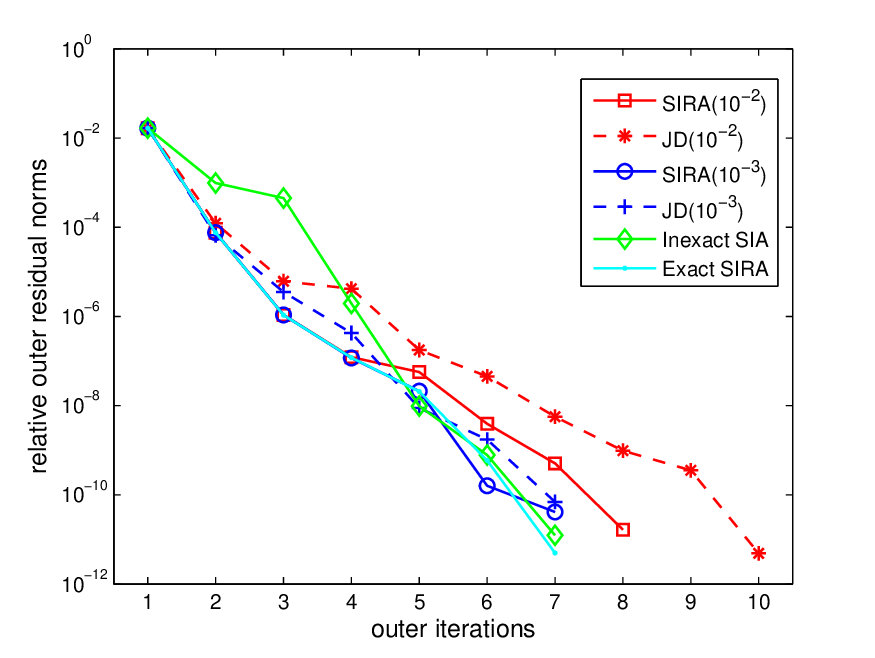}
\end{minipage}
\begin{minipage}{7.5cm}
\includegraphics[height=5.8cm]{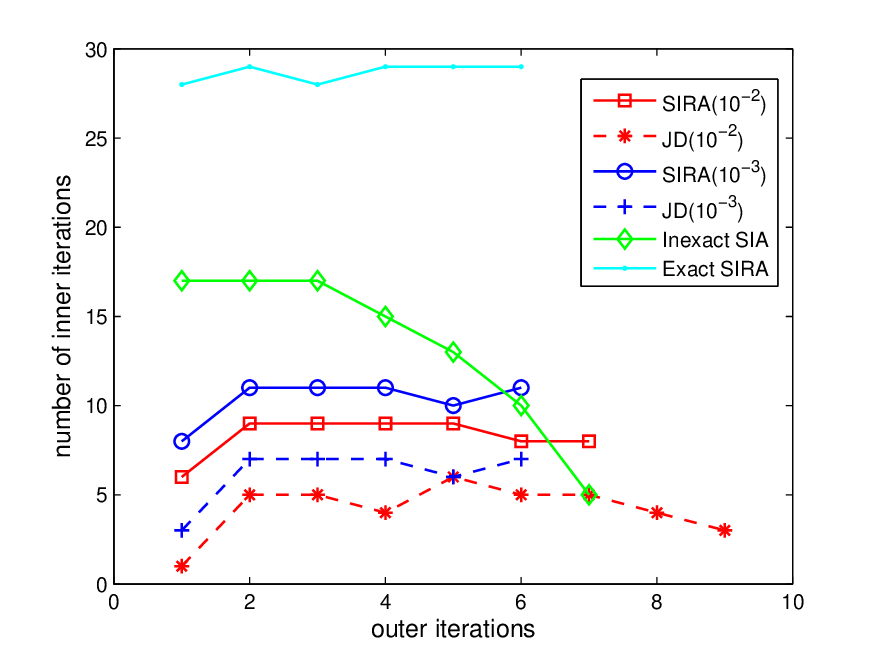}
\end{minipage}
\caption{\emph{Example 2. sherman5 with $\sigma=0$.
Left: relative outer residual norms versus outer iterations.
Right: the numbers of inner iterations versus outer iterations.}}
\label{fig_sherman5}
\end{figure}

\begin{table}[!htb]
\begin{center}
\begin{tabular}{rrrrrrrrrr}
\hline
Algorithm && $I_{inn}$ & $I_{out}$ & $I_{0.1}$ && $T_1$ & $T_2$ & $T_3$ & $T_4$ \\
\hline\hline
SIRA($10^{-2}$) && $58$  & $8$ & $0$ && $7$ & $32$ & $483$ & $1713$ \\
JD($10^{-2}$)   && $38$  & $10$& $0$ && $8$ & $44$ & $483$ & $1425$\\ \hline
SIRA($10^{-3}$) && $62$  & $7$ & $0$ && $5$ & $25$ & $483$ & $1820$ \\
JD($10^{-3}$)   && $37$  & $7$ & $0$ && $5$ & $24$ & $483$ & $1259$\\ \hline
SIRA($10^{-4}$) && $74$  & $7$ & $0$ && $5$ & $26$ & $483$ & $2174$\\
JD($10^{-4}$)   && $48$  & $7$ & $0$ && $5$ & $26$ & $483$ & $1567$\\ \hline
Inexact SIA     && $94$  & $7$ & $-$ && $4$ & $12$ & $483$ & $2821$\\
\textquotedblleft Exact\textquotedblright\ SIRA
                && $172$ & $7$ & $-$ && $6$ & $29$ & $484$ & $6583$\\ \hline \hline
\end{tabular}
\caption{\emph{Example 2. sherman5 with $\sigma=0$ (The unit of $T_1\sim T_4$ is
$0.0001$ second).}}
\label{tab_sherman5}
\end{center}
\end{table}

We see from the left part of Figure~\ref{fig_sherman5} that the inexact SIRA, JD
and SIA behaved like the exact SIRA very much and used very comparable outer
iterations. They mimic the exact SIRA better for $\tilde{\varepsilon}=10^{-3},10^{-4}$
than for $\tilde{\varepsilon}=10^{-2}$. The table also tells us that a smaller
$\tilde{\varepsilon}<10^{-3}$ is definitely not necessary as it could not reduce the
number of outer iterations and meanwhile consumed more inner iterations.
The results confirm our theory and indicate that our selection of
$\tilde{\varepsilon}$ and $\varepsilon$ worked very well.
It is obvious that, as far as outer iterations are concerned,
all the algorithms converged quickly and smoothly.

For the overall efficiency, the situation is very different.
As is expected, we see from Table~\ref{tab_sherman5} and Figure~\ref{fig_sherman5}
that the exact SIRA was the most expensive and
the inexact SIA with was the second most expensive, as the $I_{inn}$ and
the total computing time indicated.
The exact SIRA used $28\sim29$ inner iterations per outer iteration,
and the inexact SIA used $17$ inner iterations at each of the first $3$
outer iterations where the accuracy of approximate eigenpairs was poor
and the inner linear systems must be solved with high accuracy.
As the approximate eigenpairs started converging, the relaxation strategy
took effect and the inner linear systems were solved with decreasing accuracy,
so that the numbers of inner iterations became increasingly smaller
as outer iterations proceeded. In contrast, the inexact SIRA and JD were
much more efficient than the inexact SIA, they used much fewer inner iterations
and computing time than the latter
and were roughly one and a half to two times as fast as the inexact SIA.
Furthermore, we observe that the inexact JD and SIRA used quite
few and almost constant inner iterations per outer iteration for each
$\tilde{\varepsilon}$, respectively, but the former was more effective than
the latter. This may be due to the better conditioning of the
coefficient matrix in the correction equation of JD.

Also, we observe from Table~\ref{tab_sherman5}
that the time $T_4$ of solving preconditioned inner linear systems
dominates the total CPU time and on the other hand
the construction of preconditioners is the second most expensive.
So solving inner linear systems overwhelms is much more than the others,
and both $I_{inn}$ and the sum of $T_4$ and $T_3$ reflect
the overall efficiency of each algorithm very well.

\smallskip

\noindent\textbf{Example 3.} This problem arises from computational fluid dynamics
and the test matrix af23560 of $n=23560$ is from transient stability analysis of
Navier-Stokes solvers \cite{matrixmarket}. We want to find the eigenvalue nearest
to $\sigma=0$. The computed eigenvalue is $\lambda\approx-0.2731$.
The preconditioner $\bfM$ is obtained by the sparse
incomplete LU factorization of $\bfA-\sigma \bfI$ with $setup.droptol=0.01$;
see Table~\ref{tab_af23560} and Figure~\ref{fig_af23560} for the results.

\begin{figure}[!htb]
\begin{minipage}{7.5cm}
\includegraphics[height=5.8cm]{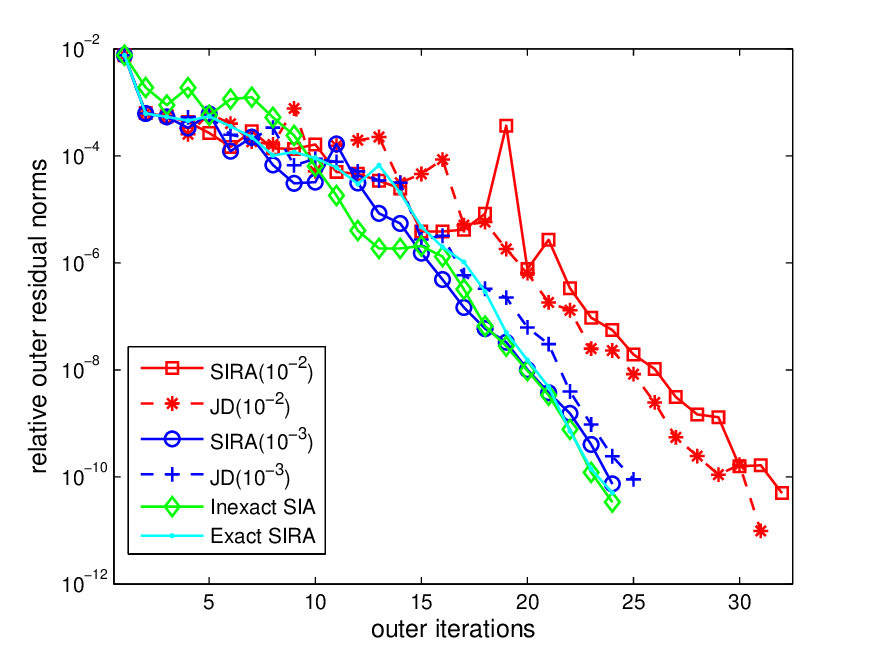}
\end{minipage}
\begin{minipage}{7.5cm}
\includegraphics[height=5.8cm]{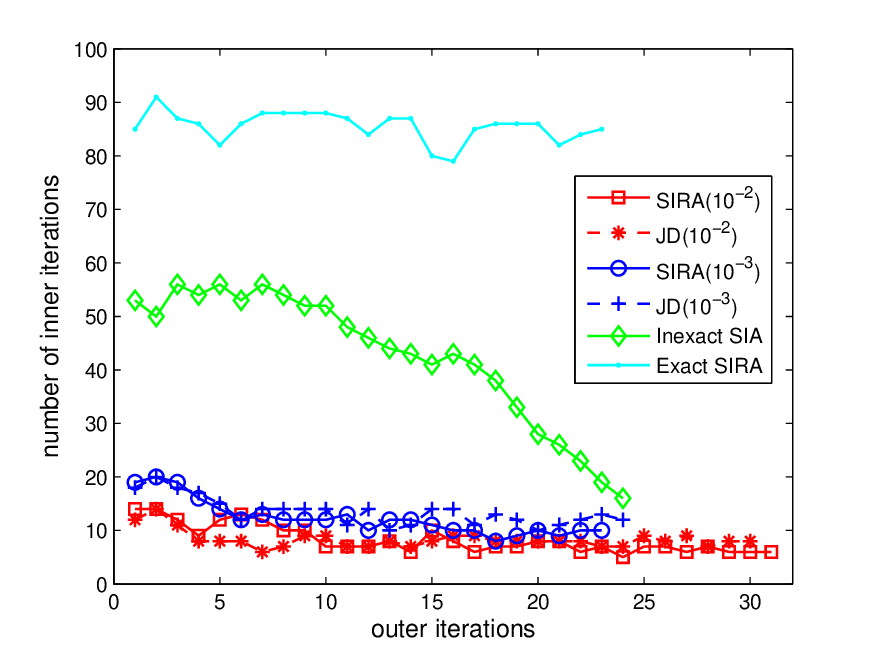}
\end{minipage}
\caption{\emph{Example 3. af23560 with $\sigma=0$.
Left: outer residual norms versus outer iterations.
Right: the numbers of inner iterations versus outer iterations.}}
\label{fig_af23560}
\end{figure}

\begin{table}[!htb]
\begin{center}
\begin{tabular}{rrrrrrrrrr}
\hline
Algorithm && $I_{inn}$ & $I_{out}$ & $I_{0.1}$ && $T_1$ & $T_2$ & $T_3$ & $T_4$ \\
\hline\hline
SIRA($10^{-2}$) && $258$  & $32$  & $19$ && $1$ & $68$ & $89$ & $1130$ \\
JD($10^{-2}$)   && $250$  & $31$  & $23$ && $1$ & $63$ & $89$ & $1140$\\ \hline
SIRA($10^{-3}$) && $283$  & $24$  & $0$ && $1$ & $37$ & $89$ & $1316$ \\
JD($10^{-3}$)   && $324$  & $25$  & $0$ && $1$ & $35$ & $89$ & $1519$\\ \hline
SIRA($10^{-4}$) && $429$  & $23$  & $0$ && $1$ & $35$ & $89$ & $2058$\\
JD($10^{-4}$)   && $400$  & $23$  & $0$ && $1$ & $32$ & $89$ & $1888$\\ \hline
Inexact SIA     && $1025$ & $24$  & $-$ && $1$ & $8$  & $89$ & $4232$\\
\textquotedblleft Exact\textquotedblright\ SIRA
                && $1967$ & $24$  & $-$ && $1$ & $31$ & $89$ & $8664$\\ \hline\hline
\end{tabular}
\caption{\emph{Example 3. af23560 with $\sigma=0$ (The unit of $T_1\sim T_4$ is
$0.01$ second).}}
\label{tab_af23560}
\end{center}
\end{table}

Compared with Examples 1--2, we see from both
Table~\ref{tab_af23560} and Figure~\ref{fig_af23560}
that for this problem all the algorithms used considerably more outer iterations
$I_{out}$ but $I_{inn}$ increases more rapidly than $I_{out}$ does. So this problem
was considerably more difficult than the previous two ones. The difficulty
is two-fold: the eigenvalue problem itself and the inner linear systems involved
in the algorithms. The second difficulty means that $T_4$ is more dominant
than it for Examples 1--2. Moreover,
we see that $T_4$ is much more than the corresponding $T_3$, the setup
time of the preconditioner. As as whole, $I_{inn}$ and the time of solving
inner linear systems reflect the overall efficiency of an algorithm more accurately.

In this example, the case that $\varepsilon=0.1$ occurred at about
$60\%$ and $75\%$ of outer iterations in SIRA($10^{-2}$) and JD($10^{-2}$),
respectively. Regarding outer iterations, we observe from
Figure~\ref{fig_af23560} that for $\tilde{\varepsilon}=10^{-3}$ the
inexact SIRA, JD and SIA behaved like
the exact SIRA very much. For the bigger $\tilde{\varepsilon}=10^{-2}$,
the inexact SIRA and SIA used more outer iterations and did not mimic
the exact SIRA well. It is amazing that SIRA($10^{-4}$) and JD($10^{-4}$) used
one less outer iteration than the exact SIRA.
Again, the results confirmed our theory and
showed that a low or modest accuracy $\tilde{\varepsilon}=10^{-3}$ is enough,
a looser $\tilde{\varepsilon}=10^{-2}$ worked quite well and only a little bit
more outer iterations were needed for it.

For the overall efficiency, the inexact SIA was better than the exact SIRA
but much inferior to the inexact SIRA and JD.
Actually, as $I_{inn}$ and $T_4$ show, the inexact SIRA and JD
with $\tilde{\varepsilon}=10^{-2},10^{-3}$
were twice to almost four times as fast as the inexact SIA.
Although SIRA($10^{-2}$) and JD($10^{-2}$) used more outer iterations
than the others, they were the most efficient
in terms of both $I_{inn}$ and $T_4$.
The exact SIRA used roughly $85$ inner iterations per outer iteration.
The inexact SIA used many inner iterations and needed to solve inner
linear systems with high accuracy for most of the outer iterations.
Even after the relaxation strategy played a role,
it still used much more inner iterations than the inexact SIRA and
JD with $\tilde{\varepsilon}=10^{-2},10^{-3}$ at each outer iteration.
Although SIRA($10^{-4}$) and JD($10^{-4}$) behaved like
the exact SIRA best and won all the others in terms of $I_{out}$,
the overall efficiency of them was not as good as that of the
the inexact methods with bigger $\tilde{\varepsilon}$.
We find that, for the same accuracy $\tilde{\varepsilon}$,
the inexact SIRA and JD solved the linear systems with slowly varying
inner iterations at each outer iteration. This is expected as
the accuracy requirements of inner iterations were almost
the same. In terms of $I_{inn}$ and $T_4$, we also observe
from Table~\ref{tab_af23560} that the inexact SIRA and JD were
equally effective and had very similar efficiency.

Still, similar to Examples 1--2, we see from $T_1\sim T_4$ that solving
preconditioned inner linear systems is the most expensive and dominates
the overall efficiency of each algorithm, while the construction
of preconditioners overwhelms the solutions of small eigensystems as
well as the generations of orthonormal basis and projected matrices.

\smallskip

\noindent\textbf{Example 4.} This unsymmetric eigenvalue problem dw8192 of $n=8192$
arises from dielectric channel waveguide problems \cite{matrixmarket}.
We are interested in the eigenvalue nearest to the complex target $\sigma=0.01\imag$.
The computed eigenvalue is
$\lambda\approx3.3552\times10^{-3}+1.1082\times10^{-3}\imag$
The preconditioner $\bfM$ is obtained by the sparse
incomplete LU factorization of $\bfA-\sigma \bfI$ with $setup.droptol=0.001$.
Table~\ref{tab_dw8192} displays the results.

\begin{table}[!htb]
\begin{center}
\begin{tabular}{rrrrrrrrrr}
\hline
Algorithm && $I_{inn}$ & $I_{out}$ & $I_{0.1}$ && $T_1$ & $T_2$ & $T_3$ & $T_4$ \\
\hline\hline
SIRA($10^{-2}$) && $312$  & $99$ & $82$ && $12$ & $53$ & $3$ & $129$ \\
JD($10^{-2}$)   && $276$  & $93$ & $81$ && $11$ & $47$ & $3$ & $144$\\ \hline
SIRA($10^{-3}$) && $386$  & $87$  & $0$ && $8$  & $41$ & $3$ & $144$ \\
JD($10^{-3}$)   && $428$  & $94$  & $1$ && $11$ & $47$ & $3$ & $192$\\ \hline
SIRA($10^{-4}$) && $466$  & $71$  & $0$ && $4$  & $26$ & $3$ & $171$\\
JD($10^{-4}$)   && $451$  & $70$  & $0$ && $4$  & $25$ & $3$ & $183$\\ \hline
Inexact SIA     && $1663$ & $86$  & $-$ && $7$  & $8$  & $3$ & $616$\\
\textquotedblleft Exact\textquotedblright\ SIRA
                && $1940$ & $66$  & $-$ && $3$  & $21$ & $3$  & $741$\\ \hline\hline
\end{tabular}
\caption{\emph{Example 4. dw8192 with $\sigma=0.01\imag$ (The unit of $T_1\sim T_4$ is
$0.1$ second).}}
\label{tab_dw8192}
\end{center}
\end{table}

As far as the eigenvalue problem is concerned,
Table~\ref{tab_dw8192} clearly indicates that this problem is much more difficult
than Examples 1--3 since all the algorithms used much more outer iterations
to achieve the convergence than those needed for Examples 1--3.
But our inexact SIRA and JD algorithms still worked very well.
The inexact SIRA and JD with $\tilde{\varepsilon}=10^{-4}$
behaved more like the exact SIRA than with $\tilde{\varepsilon}=10^{-3}$ and
$\tilde{\varepsilon}=10^{-2}$. Therefore, we can infer that a smaller
$\tilde{\varepsilon}<10^{-4}$ is not necessary and
cannot improve the behavior of the inexact SIRA and JDl; it will
make the inexact methods use almost the same
outer iterations as the exact SIRA but consume more inner iterations.
Furthermore, we have observed the inexact SIA did not mimic the exact
SIRA very well as it used considerably more outer iterations than the exact SIRA.

For the overall efficiency, Table~\ref{tab_dw8192} exhibited similar features to
those in all the previous tables for Examples 1--3.
The inexact SIRA and JD were similarly effective.
Both of them were much more efficient than the inexact SIA and actually
three to five times as fast as the latter, in terms of both $I_{inn}$ and
the total computing time.

Since this problem is difficult, we turn to use restarted SIRA and JD algorithms,
Algorithms 3--4, to solve it with the maximum $\bfM_{\max}=30$ outer iterations
allowed during each cycle. We also test the implicitly
restarted inexact SIA method \cite{spence2009ia,xueelman10} with
the same $\bfM_{\max}=30$ and make a comparison
of all the restarted algorithms.
Table~\ref{tab_dw8192_restart} lists the results obtained by the restarted
inexact SIRA, JD and SIA as well as the restarted exact SIRA,
where $I_{restart}$ denotes the number of restarts used, i.e.,
the number of the cycles of Algorithms 1--2 for the given $\bfM_{\max}$.
Figure~\ref{fig_dw8192_restart} depicts the convergence curve of all
the restarted algorithms and the curve of $I_{inn}$ versus $I_{restart}$,
in which the zeroth restart in abscissa denotes the first cycle of Algorithms 3--4 and
corresponds to the first restart in the left figure.

\begin{figure}[!htb]
\begin{minipage}{7.5cm}
\includegraphics[height=5.8cm]{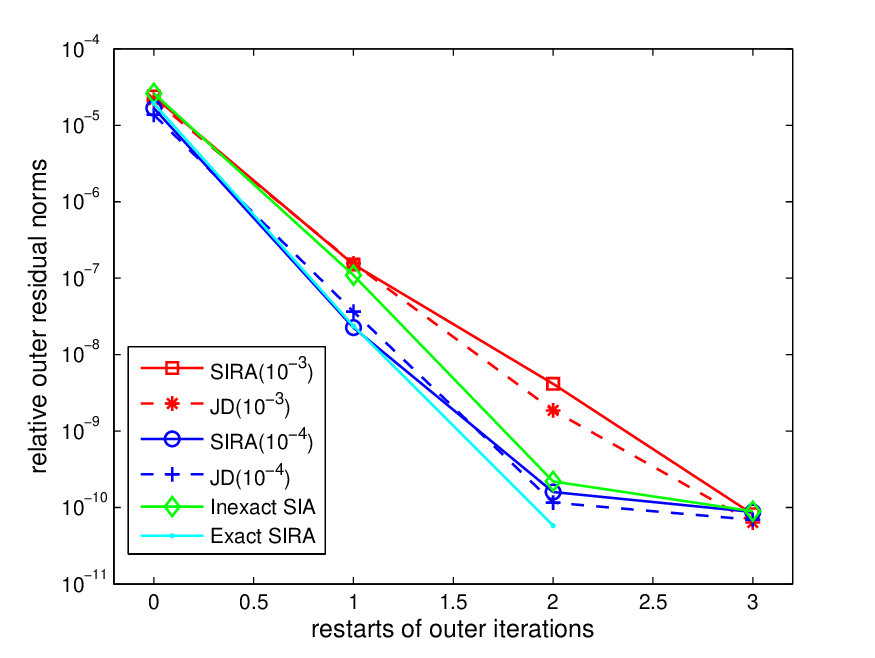}
\end{minipage}
\begin{minipage}{7.5cm}
\includegraphics[height=5.8cm]{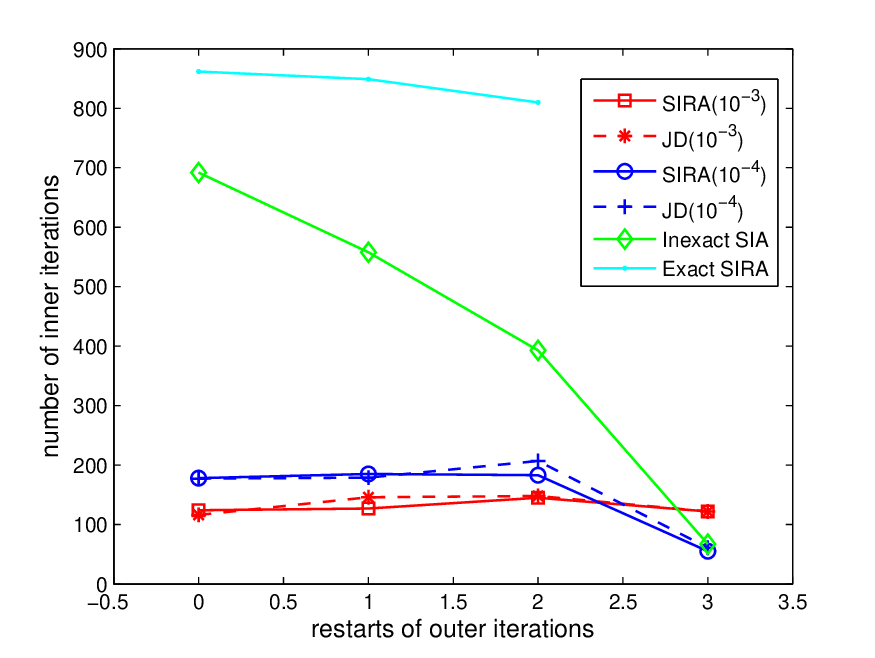}
\end{minipage}
\caption{\emph{Example 4. Restarted algorithms with $\bfM_{\max}=30$.
Left: outer residual norms versus outer iterations.
Right: the numbers of inner iterations versus restarts.}}
\label{fig_dw8192_restart}
\end{figure}
\begin{table}[!htb]
\begin{center}
\begin{tabular}{rrrrrrrrrrr}
\hline
Algorithm && $I_{inn}$ & $I_{restart}$ & $I_{out}$ & $I_{0.1}$ && $T_1$ & $T_2$ & $T_3$ & $T_4$ \\
\hline\hline
SIRA($10^{-2}$) && $876$  & $8$ & $250$ & $142$ && $2$ & $32$ & $3$ & $346$ \\
JD($10^{-2}$)   && $578$  & $5$ & $175$ & $98$  && $2$ & $23$ & $3$ & $284$\\ \hline
SIRA($10^{-3}$) && $518$  & $3$ & $115$ & $0$   && $1$ & $15$ & $3$ & $203$ \\
JD($10^{-3}$)   && $532$  & $3$ & $117$ & $1$   && $1$ & $15$ & $3$ & $234$\\ \hline
SIRA($10^{-4}$) && $601$  & $3$ & $100$ & $0$   && $1$ & $12$ & $3$ & $222$\\
JD($10^{-4}$)   && $624$  & $3$ & $98$  & $0$   && $1$ & $12$ & $3$ & $262$\\ \hline
Inexact SIA     && $1710$ & $3$ & $95$  & $-$   && $1$ & $3$  & $3$ & $602$\\
\textquotedblleft Exact\textquotedblright\ SIRA
                && $2521$ & $2$ & $89$  & $-$ && $1$ & $11$ & $3$  & $971$\\ \hline\hline
\end{tabular}
\caption{\emph{Example 4. Restarted algorithms with $\bfM_{\max}=30$ (The unit of $T_1\sim T_4$ is
$0.1$ second).}}
\label{tab_dw8192_restart}
\end{center}
\end{table}

It is seen from Table~\ref{tab_dw8192_restart} and the left part of
Figure~\ref{fig_dw8192_restart} that all the algorithms other than
SIRA($10^{-2}$) and JD($10^{-2}$)
solved the problem very successfully with no more than
three restarts used and the convergence processes were very smooth.
The restarted inexact SIA behaved like the restarted exact SIRA well
but not so well as the restarted SIRA and JD with $\tilde{\varepsilon}=10^{-4}$,
which behaved very like the restarted exact SIRA in the first two restarts
and almost converged to our prescribed convergence accuracy at the second restart.

We also find that, compared with Table~\ref{tab_dw8192_restart},
the restarted SIRA($10^{-4}$), JD($10^{-4}$) and exact SIRA performed excellently
since $I_{out}$'s used by them were very near to the ones used by their corresponding
non-restarted versions, respectively. For the restarted SIRA($10^{-2}$) and JD($10^{-2}$),
the case that $\varepsilon=0.1$ occurred at $50\%$ of outer iterations.
They did not mimic the exact SIRA well and used considerably more outer
iterations than the inexact SIRA and JD
with $\tilde{\varepsilon}=10^{-3}$ and
$\tilde{\varepsilon}=10^{-4}$. So $\tilde{\varepsilon}=10^{-2}$ is not a good choice
for the restarted inexact SIRA and JD for this example, though $I_{inn}$
and the total computing time are not so considerably more than those used
by the algorithms with $\tilde\varepsilon=10^{-3},\ 10^{-4}$.

Regarding the overall performance, for given $\tilde{\varepsilon}=10^{-3}$ and
$\tilde{\varepsilon}=10^{-4}$, the restarted SIRA and JD algorithms performed
very similarly and were about more than twice as fast as the restarted
inexact SIA, in terms of both $I_{inn}$ and the total computing
time (actually $T_4$ now). During the last cycle, the restarted inexact
SIRA($10^{-4}$) and JD($10^{-4}$) had already achieved the convergence at the
tenth and eighth outer iteration, respectively.
So we stopped the algorithm at that step and actually solved only about a third of
twenty-nine inner linear systems needed to solve in each of the previous cycles.
As a result, the number of inner iterations needed in the last circle was
also about a third of that needed in each of the first three cycles.
This is the reason why, in the right part of Figure~\ref{fig_dw8192_restart},
the curves for the restarted SIRA($10^{-4}$) and JD($10^{-4}$) had a drastic
decrease at last restart.
As is expected, the restarted inexact SIRA and JD algorithms used almost constant
inner iterations for the same $\tilde{\varepsilon}$ per restart,
while the inexact SIA used fewer and fewer inner iterations as outer iterations
converged. The figure clearly shows that the restarted inexact SIA used much more
inner iterations than the restarted SIRA($10^{-4}$) and JD($10^{-4}$) at each
of the first three cycles.

We see from Tables~\ref{tab_dw8192}--\ref{tab_dw8192_restart}
that for this example the dominant cost is still paid to the
solutions of preconditioned inner linear systems but unlike Examples 1--3 the
construction of preconditioners is very cheap and negligible, compared
with $T_4$.

In summary, it is seen from all the numerical experiments that both $I_inn$ and
$T_4$ are reasonable measures of overall performance of SIRA, JD and SIA algorithms.
\smallskip

We have tested some other problems.
We have also tested the algorithms when tuning is applied to our preconditioner
$\bfM$ \cite{spence2009ia}. All of them have shown that the inexact SIRA and JD
mimic the inexact SIA and the exact SIRA very well for
$\tilde{\varepsilon}=10^{-3},10^{-4}$ and use much fewer inner iterations than
the inexact SIA. As far as the overall efficiency is concerned,
SIRA($10^{-2}$) and JD($10^{-2}$) may work well and often use
comparably inner iterations than SIRA($10^{-3}$) and JD($10^{-3}$),
but they are likely to need considerably more outer iterations and cannot mimic
the exact SIRA well. Therefore, for the robust and general purpose,
we propose using
$\tilde{\varepsilon}\in[10^{-4},10^{-3}]$ in practice.
We have found that the tuned preconditioning has no advantage over
the usual preconditioning and is often inferior to the latter
for the linear systems involved in the inexact SIRA, JD and SIA algorithms.
For example, we have found that for Example 3 the tuned preconditioning used
about three times more inner iterations than the usual preconditioning.

\section{Conclusions and future work}\label{concl}

We have quantitatively analyzed the convergence of the SIRA and
and JD methods over one step and proved that one only needs to solve all the inner
linear systems involved in them with low or modest accuracy.
Based on the theory established, we have designed practical
stopping criteria for inner iterations of the inexact SIRA and JD. Numerical
experiments have illustrated that our theory works very well and
the non-restarted and restarted inexact SIRA and JD algorithms
behave very like the non-restarted and restarted
exact SIRA algorithms. Meanwhile, we have confirmed that the
inexact SIRA and JD algorithms are similarly effective and
both of them are much more efficient than the inexact SIA
algorithms.

It is known that the (inexact) JD method with variable shifts is used
more commonly. The analysis approach proposed in this paper may be
extended to analyze the accuracy requirements of inner iterations in
the JD method with variable shifts and a rigorous general
theory may be expected. This work is in progress.

Since the harmonic projection may be more suitable to solve the
interior eigenvalue problem, it is very significant
to consider the harmonic version of SIRA. Moreover, it is known that the
standard projection, i.e., the Rayleigh--Ritz method, and its harmonic version may
have convergence problem when computing eigenvectors \cite{jia2001analysis,jia05}.
So it is worthwhile and appealing to use the refined Rayleigh--Ritz procedure
\cite{jia97,jia2001analysis} and the refined
harmonic version \cite{jia2001analysis}
for solving the large eigenproblem considered in this paper.
These constitute our future work.
\bigskip

{\bf Acknowledgements}. We thank the two referees for their comments and suggestions.

\begin{small}
\bibliographystyle{siam}

\end{small}

\end{document}